\documentclass[11pt,reqno]{amsart}
\usepackage{amsmath,amssymb,geometry,color,tikz-cd}
\usepackage[colorlinks=true, linkcolor=red!80!black, urlcolor=purple, citecolor=blue!70!black]{hyperref}
\usepackage{amssymb}
\usepackage{bbm}
\usepackage{enumitem}
\usepackage{upgreek}
\usepackage{bm}
\usepackage{stmaryrd}
\usepackage{mathrsfs}
\usepackage{cleveref}
\allowdisplaybreaks

\newcommand{\A}{\mathbb{A}}

\newcommand{\C}{\mathbb{C}}
\newcommand{\G}{\mathbb{G}}
\newcommand{\Q}{\mathbb{Q}}
\newcommand{\R}{\mathbb{R}}
\newcommand{\Z}{\mathbb{Z}}
\newcommand{\N}{\mathbb{N}}
\renewcommand{\P}{\mathbb{P}}
\newcommand{\T}{\mathbb{T}}

\newcommand{\cD}{\mathcal{D}}

\newcommand{\cF}{\mathcal{F}}

\newcommand{\cL}{\mathcal{L}}

\newcommand{\cO}{\mathcal{O}}

\newcommand{\cV}{\mathcal{V}}

\newcommand{\cX}{\mathcal{X}}
\newcommand{\cY}{\mathcal{Y}}

\newcommand{\bG}{\mathbb{G}}

\newcommand{\bL}{\mathbb{L}}

\newcommand{\isom}{\cong}

\newcommand{\Bl}{\mathrm{Bl}}

\newcommand{\s}{\mathrm{s}}

\DeclareMathOperator{\Div}{Div}

\DeclareMathOperator{\Fut}{Fut}

\DeclareMathOperator{\vol}{vol}

\DeclareMathOperator{\ord}{ord}

\DeclareMathOperator{\Supp}{Supp}

\newcommand{\bT}{\mathbb{T}}

\newcommand{\coeff}{\mathrm{coeff}}

\newcommand{\bVi}{\overline{V}_\infty}

\numberwithin{equation}{section}

\newtheorem{prop}{Proposition} [section]
\newtheorem{thm}[prop] {Theorem}

\newtheorem{cor}[prop]{Corollary}
\newtheorem{prop-def}[prop]{Proposition-Definition}
\newtheorem{convention}[prop]{Convention}

\newtheorem{lemma}[prop]{Lemma}

\newtheorem{thm-defn}[prop]{Theorem-Definition}

\theoremstyle{definition}
\newtheorem{exa}[prop] {Example}

\newtheorem{definition}[prop]{Definition}

\title{On the K-stability of blow-ups of projective bundles}
\author{Daniel Mallory}
\date{\today}

\begin{document}

\begin{abstract}
We investigate the K-stability of certain blow-ups of $\P^1$-bundles over a Fano variety $V$, where the $\P^1$-bundle is the projective compactification of a line bundle $L$ proportional to $-K_V$ and the center of the blow-up is the image along a positive section of a divisor $B$ also proportional to $L$. When $V$ and $B$ are smooth, we show that, for $B \sim_{\Q} 2L$, the K-semistability and K-polystability of the blow-up is equivalent to the K-semistability and K-polystability of the log Fano pair $(V,aB)$ for some coefficient $a$ explicitly computed. We also show that, for $B \sim_{\Q} l L$, $l \neq 2$, the blow-up is K-unstable. 
 
\end{abstract}
    
\maketitle

\section{Introduction}

Originally introduced in \cite{Tian97, Donaldson02} to provide a criterion for the existence of K\"ahler-Einstein metrics on Fano manifolds, K-stability has recently been a major area of research over the last decade. In the field of complex geometry, the work of Chen-Donaldson-Sun~\cite{CDS15} and Tian~\cite{Tian15} proved the Yau-Tian-Donaldson conjecture, showing that a K\"ahler-Einstein metric exists on a given Fano manifold precisely if and only the manifold is K-polystable. Recently, the algebraic theory of K-stability has seen remarkable progress, culminating in the discovery that K-stability provides a setting in which to construct moduli stacks of K-semistable Fano varieties with corresponding good moduli spaces of K-polystable Fano varieties, called K-moduli (see \cite{Xu20, LXZ22, Xu24}).

Since K-stability provides the framework for a theory of moduli spaces for Fano varieties, a central aspect of research in the field of K-stability is developing methods to determine when a Fano variety is indeed K-stable (or K-polystable, or K-semistable, etc.). Key results towards this end include: Tian's criterion on $\alpha$-invariants of Fano varieties, providing a sufficient condition to show K-polystability (\cite{Tian87, OS12}); the Fujita-Li criterion (\cite{Fujita16, Li17}), relating the Futaki invariants of Donaldson's algebraic reformulation of K-stability to $\beta$-invariants consisting of birational data of the variety (or, more specifically, that of divisors on birational models of the variety); the development of stability thresholds (\cite{BJ20}), relating $\beta$-invariants to the so-called $\delta$-invariant of a variety, thus relating K-stability to basis-type divisors, log canonical thresholds, and log canonical places of complements, which in turn led to an ``inverse of adjunction''-type theorem by Abban-Zhuang (\cite{AZ22}); and the equivalence of K-poly/semistability to a $G$-equivariant setting of the same (with the additional restriction of $G$ reductive for K-polystability) (\cite{DS16, LX18, LWX21, Zhuang21}).

Much of the current literature on determining the K-stability of explicit Fano varieties are in low dimensions (i.e. dimensions $2$ and $3$). In dimension 2, the K-stability of smooth del Pezzo surfaces is well-investigated (see \cite{Tian90, Cheltsov08, PW18}); as is that of singular del Pezzo surfaces (see \cite{MM93, OSS12} for example). In dimension 3, the question of K-stability has been systematically approached via the Mori-Mukai classification of Fano threefolds into 105 deformation families. For example, in \cite{ACCFKMGSSV23} they determine for which of the 105 families is a general member K-polystable. The wall-crossing phenomenon for certain log Fano  pairs of dimension at most $3$ have been studied (see e.g. \cite{ADL24, ADL23, Zhao23a, Zhao23b}). In higher dimensions, there is much less known about the K-stability of explicit Fano varieties. Much of what is known is related to hypersurfaces of projective space (see e.g. \cite{Fujita17, AZ22, AZ23}).

In this paper, we construct classes of K-polystable Fano varieties from K-polystable log Fano pairs of dimension one less, thus leading to many new examples of K-polystable Fano varieties in higher dimensions. Specifically, we relate the K-stability of Fano varieties that are blow-ups of certain $\P^1$-bundles over Fano varieties to the log K-stability of the base of the $\P^1$-bundle structure. More explicitly, let $V$ be a Fano variety of dimension $n-1$, $\cL$ an ample line bundle on $V$ such that $r \cL \sim_\Q -K_V$ for some $r>1$, and $B$ an effective divisor on $V$ such that $B \sim_\Q 2\cL$. Let $Y = \Bl_{B_\infty} \P_V(\cL \oplus \cO_V)$, where $B_\infty$ is the image of $B$ under some positive section of the $\P^1$-bundle structure of $\P_V(\cL \oplus \cO_V) \to V$.

This construction of $Y$, although it may seem rather artificial, appears with some frequency among collections of Fano varieties. For example, smooth del Pezzo surfaces of degree 6 are of this form, as are the smooth members of the families \textnumero 3.9, \textnumero3.19, and \textnumero 4.2 in the Mori-Mukai classification of Fano threefolds. See Section~\ref{sec:examples} for more examples of Fano varieties arising from this construction. 

\begin{thm}\label{main}
    Let $V$, $B$ be as above. Further, let $V$ and $B$ be smooth. Then the variety $Y$ as constructed above is K-semistable (resp. K-polystable) if and only if $(V,aB)$ is  K-semistable (resp.  K-polystable), where 
    \[
       a = a(n,r) :=  \frac{r^{n+1} - (r-1)^{n+1} - (n+1)(r-1)^n }{2(n+1)(r^n - (r-1)^n)} 
    \]

\end{thm}

We note the additional hypothesis of $V$ and $B$ being smooth. We believe that this statement holds in more generality, and this will be the focus of a future work that will also explore the implications on the K-moduli of the relevant Fano varieties that such a more general statement would imply.

We note that similar results regarding the K-stability of such varieties were obtained in \cite{CDGFKMG23}; specifically \cite[Theorem 1.10]{CDGFKMG23}, which gives a numerical criterion for the K-polystability of smooth Casagrande-Druel varieties. They further conjecture that such varieties are K-polystable if and only if the base space and the related double cover are both K-polystable. In the smooth case, we confirm their conjecture as a consequence of Theorem \ref{main}.

\begin{cor}{\cite[Conjecture 1.16]{CDGFKMG23}}\label{ceaconj}
    Let $V,B$ be as above. Further, let $V$ and $B$ be smooth.
    Let $W \to V$ be the double cover of $V$ ramified over $B$. Suppose that both $W$ and $V$ are K-polystable. Then $Y$ is also K-polystable.
\end{cor}

Our approach to Theorem \ref{main} is as follows: we first reduce the statement to that of $\bT$-equivariant K-stability with the standard $\bT$-action induced by the $\P^1$-bundle structure where $\bT = \bG_m$. Then for the reverse direction of the statement on K-semistability, we use \cite[Theorem 3.3]{AZ22} to bound the $\beta$-invariants of most $\bT$-equivariant divisors over $Y$ (specifically, those with ``vertical'' centers on $Y$) and directly compute the $\beta$-invariants of the rest (those with ``horizontal'' centers). For the forward direction, we construct an explicit destabilizing divisor on $Y$ given a destabilizing divisor on $(V,aB)$. For the K-polystability case of the statement, we directly show that, given the K-polystability of $(V,aB)$ (resp. $Y$), a divisor $E$ over $Y$ (resp. $(V,aB)$) with $\beta$-invariant $0$ must induce a product test configuration. 

Theorem~\ref{main} reduces questions of the K-stability of certain families of Fano varieties to similar questions set in one dimension smaller, in exchange for the added complexity of now dealing with the K-stability of log pairs. However, this added complexity is often already of research interest; as seen in several examples, the K-stability of pairs is worked out in the investigation of the wall-crossing phenomenon of moduli spaces and, from some perspectives, is a main focus of investigating such wall-crossing phenomena. We direct readers towards \cite{ADL24, Zhou22} for more on the wall-crossing phenomenon for K-stability.

We also note that, in the construction of $Y$, whose geometry we explore in Section~\ref{secgeo}, the choice of $B$ being $\Q$-linearly equivalent to $2\cL$ is necessary for $Y$ to be K-semistable. Analysis of the generalization of the construction when $B \sim_\Q l \cL$ for $l \neq 2$ leads to the following:

\begin{thm}\label{unstable}
    Let $Y$ be constructed as above with $V$ and $B$ smooth and $B \sim_\Q l \cL$ for $0 \leq l < r+1, l \neq 2$. Then $Y$ is K-unstable. Furthermore, either the strict transform of the image of the positive section containing $B_\infty$, denoted as $\bVi$, or the strict transform of the zero section, $V_0$, is a destabilizing divisor for $Y$.
\end{thm}

We comment that the restriction on $l$ of $0 \leq l < r+1$ is precisely the range such that $Y$ is also Fano, with the interpretation that for $l = 0$, $Y$ is simply the $\P^1$-bundle $\P_V(\cL \oplus \cO_V)$. The case $l = 0$ was previously known, see \cite[Theorem 1.3]{ZZ22}. From Theorem~\ref{unstable} we provide a simple construction of a K-unstable Fano family in each dimension. Specifically, by Theorem~\ref{unstable}, The blow-up of $\P^n$ along a codimension $2$ linear subvariety is K-unstable; see Example~\ref{exauns}. 

Finally, during the preparation for our manuscript, we learned from Linsheng Wang that a different proof of Theorem~\ref{main} can be obtained from combining our Lemma~\ref{futaki} and \cite[Theorem 1.1]{Wang24}.

\subsection{Acknowledgements} I would like to thank Yuchen Liu for suggesting the problem and for many helpful conversations regarding it.

The author was partially supported by NSF Grant DMS-2148266 and NSF CAREER Grant DMS-2237139.

\tableofcontents
\section{Preliminaries}
First we review several definitions and theorems regarding the K-stability of Fano varieties. We work over $\mathbb{C}$ for the entirety. A log Fano pair is a klt projective pair $(X,D)$ such that $-(K_X +D) $ is ample. We recall the definitions of $\beta$-invariants and product-type divisors; these contain all of the necessary data to define the K-semistability and K-polystability of log Fano pairs.

\begin{definition}[$\beta$-invariant, \cite{Fujita19}, \cite{Li17}]
    For $(X,D)$ a log Fano pair and $E$ some divisor over $X$ (that is to say, $E$ is a prime divisor on some birational model $\pi: Y \to X$ of $X$), we define the $\beta$-\textit{invariant} $\beta_{X,D} (E)$ as follows: Let $A_{X,D}(E)$ denote the log discrepancy of $E$ with respect to the log pair $(X,D)$, and let 
    \[
    S_{X,D}(L; E) = \frac{1}{\vol(L)} \int_0^\infty \vol(\pi^*(L) -tE ) dt
    \]
    for $L$ a big $\Q$-Cartier divisor on $X$ where $\vol(L) = \lim_{k\to \infty} \frac{\dim H^0(X, L^{\otimes k})}{k^n / n!}$ denotes the \textit{volume} of $L$. When $L = -(K_X + D)$, we will usually omit the line bundle from the notation and simply write $S_{X,D}(E)$, which is often called the \textit{expected order of vanishing} of $E$ with respect to the pair $(X,D)$. Then we define $\beta_{X,D} (E) := A_{X,D}(E) - S_{X,D}(E)$.

    Since we often will pass between divisors and divisorial valuations, we wish to mention the definition of the $\beta$-invariant of a valuation $v$. By a \textit{valuation on $X$}, we mean a valuation $v: \C(X)^* \to \R$ that is trivial on $\C$. Since $X$ is projective, given a valuation $v$ on $X$ there will exist a unique point $\eta \in X$ such that $v \geq 0$ on $ \cO_{X,\eta} $ and $v > 0$ on the maximal ideal $m_{X,\eta} \subset \cO_{X,\eta}$ called the \textit{center} of $v$ on $X$. We will denote the closure of the center $\eta$ of a valuation $v$ as $c_X(v)$. For a divisor $E$ over $X$, we will write $c_X (E)$ for the center of the valuation $\ord_E$.
    
    We will define $\beta_{X,D} (v) := A_{X,D}(v) - S_{X,D}(v)$, where $A_{X,D}(v)$ is the log discrepancy of the valuation (see \cite[Proposition 5.1]{JM12} and \cite[Theorem 3.1]{BdFFU15}), while 
    \[
    S_{X,D}(L;v) = \frac{1}{\vol(L)} \int_0^\infty \vol(L -tv) dt
    \]
    where $\vol(L -tv) = \lim_{k\to \infty} \frac{\dim \{s \in H^0(X, L^{\otimes k}) |  v(s) \geq kt\}}{k^n / n!}$. For $E$ a divisor over $X$, we have that $S_{X,D}(L;\ord_D) = S_{X,D}(L;E)$. 
\end{definition}

\begin{definition}[test configurations]
    For a triple $(X, D, L)$ of a log Fano pair $(X,D)$ with $X$ dimension $n$ and $L$ an ample line bundle on $X$ with $L \sim -m(K_X + D)$ for some $m \in \N$, a \textit{test configuration} $(\cX, \cD, \cL)$ is a triple consisting of:
    \begin{itemize}
        \item a normal variety $\cX$;
        \item an effective $\Q$-divisor $\cD$ on $\cX$;
        \item a flat projective morphism $\pi: (\cX, \cD) \to \A^1$ such that $\cL$ is a $\pi$-ample line bundle;
        \item a $\G_m$-action on the triple $(\cX, \cD, \cL)$ such that $\pi$ is equivariant with respect to the standard multiplicative action of $\G_m$ on $\A^1$;
        \item a $\G_m$-equivariant isomorphism $\nu: (\cX \setminus \cX_0,  \cD|_{\cX \setminus \cX_0}, \cL|_{\cX \setminus \cX_0}) \to (X,D,L) \times (\A^1\setminus\{0\})$.
    \end{itemize}

    We say a test configuration is a \textit{product test configuration} if the above isomorphism $\nu$ extends to an isomorphism $(\cX, \cD, \cL) \cong (X,D,L) \times \A^1$. The restriction of the valuation $\ord_{\cX_0}$ to $\C(X) \subset \C(\cX) \cong \C(X)(t)$ yields a divisorial valuation on $X$; we will call the divisor $E$ such that $\ord_E = \ord_{\cX_0}|_{\C(X)}$ the \textit{divisor associated} to the test configuration $(\cX, \cD, \cL)$.
\end{definition}

\begin{definition}[product-type divisors]
\label{defprdtype}
    Let $(X,D)$ be a log Fano pair and let $E$ be a divisor over $X$. Let 
    \[
    \text{Rees}(E) := \bigoplus_{m \in \N} \bigoplus_{p \in \Z} \{s \in H^0 (-mK_X ) | \ord_E (s) \geq p\} 
    \]
    be the \textit{Rees algebra} of $E$. Suppose $\text{Rees}(E)$ is finitely generated (in such case, we call $E$ a \textit{dreamy divisor}). We define the \textit{test configuration associated to $E$} to be   
    \[
        TC(E) := \text{Proj}_X \left(\text{Rees}(E)\right)
    \]
    
    We say $E$ is a \textit{product-type divisor} over $X$ if $TC(E)$ is a product test configuration.
\end{definition}

\begin{definition}(K-stability, \cite{Fujita19}, \cite{Li17}, \cite[Theorem E]{LX18}, \cite[Theorem 1.4]{LWX21}, \cite[Theorem 1.1]{Zhuang21})
\label{defstab}
    Let $(X,D)$ be a log Fano pair. Suppose we have an action of an algebraic torus $\bT = \G_m^r$ on $(X,D)$. Then, we say that:
    \begin{itemize}
        \item $(X,D)$ is \textit{K-semistable}, if, for all $\bT$-invariant divisors $E$ over $X$, $\beta_{X,D}(E) \geq 0$. 
        \item  Then $(X,D)$ is \textit{K-polystable}, if $X$ is K-semistable and, for all $\bT$-invariant divisors $E$ over $X$, $\beta_{X,D}(E) = 0$ implies that $E$ is a product-type divisor.
        \item $(X,D)$ is \textit{K-unstable} if it is not K-semistable.  
    \end{itemize}
\end{definition}

\begin{definition}[filtrations]  
    For a finite dimensional vector space $V$, we will say a \textit{filtration} (or $\R$\textit{-filtration}) $\cF$ on $V$ is a collection of subspaces $\cF^\lambda V \subseteq V$ for $\lambda \in \R$ such that $\cF^{\lambda_0} V = V$ for some $\lambda_0 \in \R$, $\cF^{\lambda_1} V = 0$ for some $\lambda_1 \in \R$, $\lambda \geq \lambda'$ implies $\cF^{\lambda'}V \subseteq \cF^\lambda V$, and that the collection is left-continuous, which is to say that $\cF^{\lambda -\epsilon} V = \cF^\lambda$ for all $\lambda$ and sufficiently small $0< \epsilon$. 

    A valuation $v$ on $V$ has a naturally associated filtration $\cF_v V$ where $\cF_v^\lambda V = \{x \in V | v(x) \geq \lambda\}$. In particular, given a linear system $V \subseteq H^0(X,L)$ and a divisor $D$ on $X$, there is a naturally associated filtration on $V$ given by the filtration associated to $\ord_D$.

    Given a filtration $\cF$ on a vector space $V$, we say a basis $\{s_0, \ldots, s_m\} \subset V$ of $V$ is \textit{compatible} with $\cF$ if, for all $\lambda \in \R$, $\cF^\lambda V$ is spanned by some subset of the $s_i$. 
\end{definition}

Next we recall details of $\delta$-invariants, as related to both $\beta$-invariants and basis-type divisors, to the end of recalling \cite[Theorem 3.3]{AZ22}, which will be a key lemma in the proof of our main result. 

\begin{definition} [$\delta$-invariants, \cite{BJ20}]  
    For a log Fano pair $(X,D)$, a subvariety $Z \subseteq X$, and a big $\Q$-Cartier divisor $L$ on $X$, we define the $\delta$-invariant of the triple $(X,D,L)$ along $Z$ as
    \[
    \delta_Z (X,D,L) = \inf_{E ,  Z \subseteq c_X(E)}  \frac{A_{X,D}(E)}{S_{X,D}(L;E)}
    \]
    where the infimum is taken over divisors $E$ over $X$ whose centers on $X$ contain $Z$. As above, when $L = - (K_X +D)$, we will often omit $L$ from the notation. Likewise, when $Z = X$, we will omit the subscript. From Definition~\ref{defstab}, we see that $(X,D)$ is K-semistable precisely when $\delta(X,D) \geq 1$. 
\end{definition}

    There is an alternate approach to determining $\delta$-invariants due to \cite{BJ20}; towards that end, we recall basis-type divisors.

\begin{definition} (basis-type divisors,  \cite[Definition 0.1]{FO18})
    Let $L$ be a big $\Q$-Cartier $\Q$-divisor on $(X,D)$, and $V$ some linear series $V \subset H^0(X, mL)$. For a global section $s \in V$, let $\{s = 0\}$ denote the divisor $\Div(s) +mL$. Then, a \textit{basis-sum divisor} of $V$ is a divisor $\Delta$ of the form
    \[
    \Delta = \sum_{i = 0}^{N_m} \{s_i = 0\} 
    \]
    where $N_m = h^0(X,mL)$ and $\{s_0, \ldots s_{N_m}\}$ form a basis of $V$. We define a divisor $\Delta'$ to be an \textit{$m$-basis type $\Q$-divisor} of the line bundle $L$ if $\Delta' = \frac{1}{mN_m} \Delta$ for $\Delta$ a basis-sum divisor of $H^0(X,mL)$.

    Given a valuation $v$ on $X$, we say that an $m$-basis type $\Q$-divisor $\Delta$ is \textit{compatible with $v$} if the basis $\{s_0, \ldots s_{N_m}\}$ is compatible with the filtration on $V$ associated to $v$. If $v = \ord_D$ for a divisor $D$ over $X$, then we will say $\Delta$ is \textit{compatible with $D$}.  

    Let $\delta_{Z,m} (X,D,L) = \sup\{ t \geq 0 | (X,D +t \Delta) \text{ is log canonical near } Z \} $
    where the supremum is over all $m$-basis type $\Q$-divisors $\Delta$ of $L$. 
    
    By \cite[Theorem 4.4]{BJ20}, we have that $\delta_Z (X,D,L) = \lim_{m\to \infty} \delta_{Z,m} (X,D,L)$, thus allowing us to approach K-stability via basis-type divisors.

    We also can approximate $S_{X,D}$. Let
    \[
    S_m (L;v) = \sup_D v (D)
    \]
    where the supremum is taken over $m$-basis type $\Q$-divisors of $L$. Then, by \cite[Corollary 3.6]{BJ20}, we have that $S_{X,D} (L;E) = \lim_{m\to\infty} S_m(L;E)$. 
\end{definition}

\begin{definition}(multi-graded linear systems, \cite[Definition 2.11]{AZ22})
    Let $L_1, \ldots L_k$ be big $\Q$-Cartier $\Q$-divisors on $(X,D)$. A $\textit{multi-graded linear system}$ $V_{\Vec{\bullet}}$ graded by $\N^k$ on $(X,D)$ consists of a set of subspaces 
    \[
        V_{\Vec{u}} \subset H^0 (X, u_1 L_1 + \ldots u_k L_k)
    \]
    for all $\Vec{u} \in \N^k$ such that $V_{\Vec{0}} = \C$, and, for any $\Vec{u}_1, \Vec{u}_2 \in \N$, $V_{\Vec{u_1}} \cdot V_{\Vec{u_2}} \subset V_{\Vec{u_1} + \Vec{u_2}}$. Denote
    \[
    N_m = \sum_{\Vec{u} \in \{m\} \times \N^r } h^0 (X, V_{\Vec{u}})
    \]
    We can now analogously define an $m$-basis type $\Q$-divisor of a $N^{k+1}$-graded linear system $V_{\Vec{\bullet}}$ to be a divisor $\Delta$ of the form
    \[
    \Delta = \frac{1}{ m N_m } \sum_{\Vec{u} \in \{m\} \times \N^r} \Delta_{\Vec{u}}
    \]
    where $\Delta_{\Vec{u}}$ is a basis-sum divisor of $V_{\Vec{u}}$. We also analogously define $S_m (V_{\Vec{\bullet}};E)$, $S_{X,D}(V_{\Vec{\bullet}};E)$, $\delta_m(X,D,V_{\Vec{\bullet}})$ and $\delta(X,D,V_{\Vec{\bullet}})$, and note that, for a big line bundle $L$, we have that 
    \[
    \delta(X,D,W_{\bullet}) = \delta(X,D,L)
    \]
    for the $\N$-graded linear system $W_{\bullet}$ with $W_m = H^0(X,L^{\otimes m})$. 
\end{definition}

\begin{definition}(refinements of linear systems, \cite[Example 2.1]{AZ22}) \label{refine}
    Let $(X,D)$ be a log Fano pair, $V_{\Vec{\bullet}}$ an $\N^k$-graded linear system on $X$ for some big $\Q$-Cartier $\Q$-divisors $L_1, \ldots L_{k}$, and $E \subset Y \xrightarrow{\pi} X$ a divisor over $X$. We define \textit{the refinement of $V$ by $E$} to be the $\N^{k+1}$-graded linear series $W_{\Vec{\bullet}}$ associated to the big $\Q$-Cartier $\Q$-divisors $L_1|_E, \ldots L_{k}|_E, -E|_E$, with
    \begin{align*}
        W_{\Vec{u},j} &= \text{Im}\left( \cF_j V_{\Vec{u}} \to H^0 (X, \pi^*(u_1 L_1 + \ldots + u_k L_k) - jE) \right. \\
        & \left. \xrightarrow{\rho} H^0(E, \pi^*(u_1 L_1 + \ldots + u_k L_k)|_E - j E|_E) \right)
    \end{align*}
    where the map $\rho$ is the restriction map on global sections.
\end{definition}

The invariant data of a multi-graded linear system has a close relation with that of its refinement, as is seen from ~\cite[Theorem 3.3]{AZ22}; this theorem will in turn be a key lemma in the argument for the semistable version of Theorem~\ref{main}.

\begin{thm} \cite[Theorem 3.3]{AZ22} 
    Let $(X,D)$ be a log pair, $V_{\Vec{\bullet}}$ a multi-graded linear system on $X$, $E \subset Y \xrightarrow{\pi} X$ a primitive divisor over $X$, and $Z \subset X$ a subvariety. Let $Z_0$ denote an irreducible component of $Z \cap c_X(E) $, let $D_Y$ denote the strict transform of $D$ on $Y$, and let $D_F$ denote the different (that is, the divisor on $E$ such that $(K_Y + D_Y + E)|_E = K_E + D_E$. Let $W_{\Vec{\bullet}}$ denote the refinement of $V_{\Vec{\bullet}}$ by $E$. Then, if $Z \subset c_X(E)$, we have the following inequality of $\delta$-invariants
    \[
        \delta_Z (X,D, V_{\Vec{\bullet}}) \geq \min \left\{ \frac{A_{X,D} (E)}{S_{X,D}(V_{\Vec{\bullet}})}, \inf_{Z'} \delta_{Z'} (E, D_E, W_{\Vec{\bullet}}) \right\}
    \]
    and, if $Z \not\subset c_X(E)$,
    \[
        \delta_Z (X,D, V_{\Vec{\bullet}}) \geq \inf_{Z'} \delta_{Z'} (E, D_E, W_{\Vec{\bullet}})
    \]
    with both infima over subvarieties $Z' \subset Y$ such that $\pi(Z') = Z_0$. 
\end{thm}

Due to the natural torus action on $\P^1$-bundles, we also recall the equivalence between K-stability and the notion of $\bT$-equivariant K-stability.

\begin{thm}\cite[Theorem E]{LX18}, \cite[Theorem 1.4]{LWX21}, \cite[Theorem 1.1]{Zhuang21}
\label{equiv}
Let $(X,D)$ be a log Fano pair with a $\bT$-action for $\bT \cong \G_m^r$ a torus group. Then the K-semistability (resp. K-polystability) of $(X,D)$ is equivalent to the $\bT$-equivariant K-semistability (resp $\bT$-equivariant K-polystability) of $(X,D)$.
\end{thm}

\section{Blow-ups of Projective Compactifications of Proportional Line Bundles}
\label{secgeo}
Let $V$ be a smooth, $(n-1)$-dimensional Fano variety. We construct an $n$-dimensional Fano variety from the data of $V$, a choice of an ample line bundle $\cL$ on $V$ such that $r \cL \sim_{\Q} -K_V$ for some proportionality constant $r>1$, and a divisor $B$ on $V$ such that $B$ is smooth and $B \sim_{\Q} l \cL $ for some  proportionality constant $0 \leq l < r+1$.

Firstly we consider $X$, the projective compactification of the total space of the line bundle $\cL$. $X$ can be geometrically realized as the projectivization $\P (\cL \oplus \cO_V)$ of the total space of the vector bundle $\cL \oplus \cO_V$. The morphism from the vector bundle structure $\cL \oplus \cO_V \to V$ induces a morphism $\phi: X \to V$, giving $X$ the structure of a $\P^1$-bundle over $V$. With this structure, we denote the image of the zero section as $V_0$. We have the following formula for the anti-canonical divisor $-K_X$ on $X$ in relation to that of $V$:
\[
    -K_X = 2H +\phi^* \left(-K_V - \text{det} (\cL \oplus \cO_V) \right) = 2H + \frac{r-1}{r}  \phi^* ( -K_V )
\]
where $H$ denotes the relative hyperplane section from the $\P^1$-bundle structure of $X$. Thus, we see that $X$ is also Fano exactly when $r > 1$. We also compute the anti-canonical volume of $X$, that is, the top intersection power of $-K_X$, as a function of $r, n,$ and the anti-canonical volume of $V$:

\begin{lemma}
    $X$ is a Fano variety. Moreover, the anti-canonical volume of $X$ is 
    \[
    \vol(X) := (-K_X)^n = \frac{(r+1)^n - (r-1)^n}{r^{n-1}} \vol(V).
    \]
\end{lemma}
\begin{proof}
    The anti-canonical divisor of $X$ is 
    \[
    -K_X = 2H +\phi^* (-K_V) - \text{det} (\cO_V \oplus \cL)
    \]
    where $H$ is the relative hyperplane section. Since $\cL \sim \frac{-1}{r} K_V$, the above simplifies to 
    \[
    -K_X = 2H +\frac{r-1}{r} \phi^*(-K_V)
    \]
    We also have that $H = V_0 +\phi^*(\cL)$ and thus 
    \begin{align*}
    H^2 &= V_0 \cdot H + \phi^* (\cL) \cdot H \\
        &= \phi^*(\cL) \cdot H 
    \end{align*}
    and, more generally,
    \[
    H^k = \phi^*(\cL)^{k-1} \cdot H
    \]
    From this, we calculate $(-K_X)^n$
    \begin{align*}
        (-K_X)^n &= \left[ 2H + \frac{r-1}{r} \phi^*(-K_V) \right]^n \\
        &= \sum_{k=0}^n \binom{n}{k} 2^{n-k} (\frac{r-1}{r})^k H^{n-k} \cdot \phi^*(-K_V)^k \\
        &= \sum_{k=0}^{n-1} \binom{n}{k} 2^{n-k} (\frac{r-1}{r})^k \phi^*(\cL)^{n-k-1} \cdot H \cdot \phi^*(-K_V)^k \\
        &= \sum_{k=0}^{n-1} \binom{n}{k} 2^{n-k} (\frac{r-1}{r})^k (\frac{1}{r})^{n-k-1} \phi^*(-K_V)^{n-k-1} \cdot H \cdot (-K_V)^k \\
        &= \sum_{k=0}^{n-1} \binom{n}{k} 2^{n-k} \frac{(r-1)^k}{r^{n-1}} H \cdot \phi^*(-K_V)^{n-1} \\
        &= \frac{(2+r-1)^n - (r-1)^n}{r^{n-1}} \vol(V) \\
    \end{align*}
In particular, we see that $-K_X$ has strictly positive top intersection power. 

Now, we consider the intersection of $-K_X$ with effective classes of curves on $X$. In particular, since $-K_X$ is invariant under the $\bT$-action of the $\P^1$-bundle, we consider intersection products of $-K_X$ with effective classes of curves fixed by the $\bT$-action. Since $\bT$ is $1$ dimensional, a $\bT$-fixed curve is either the closure of a single orbit or a curve that is fixed point-wise. The closure of a single orbit would be a fibre under the $\P^1$-bundle structure; a curve that is fixed point-wise would be contained in the fixed locus of the $\bT$-action, which is $V_0 \sqcup V_\infty$. Thus, we concern ourselves with the following classes of curves: curves $c_0 \subset V_0$, curves $c_\infty \subset V_\infty$, and fibres $f$ of the $\P^1$-bundle structure. We have the following intersection products:
\begin{align*}
    -K_X \cdot c_0 &=  (2V_\infty +\frac{r-1}{r} \phi^*(-K_V)) \cdot c_0 = \frac{r-1}{r} \phi^*(-K_V) \cdot c_0 > 0 \\
    -K_X \cdot c_\infty &=  (2V_0 +\frac{r+1}{r} \phi^*(-K_V)) \cdot c_\infty = \frac{r+1}{r} \phi^*(-K_V) \cdot c_\infty > 0 \\
    -K_X \cdot f &=  (2H +\frac{r-1}{r} \phi^*(-K_V)) \cdot f = 2H \cdot f > 0 
\end{align*}

Thus, we see $-K_X$ is strictly nef, and thus the computation of the top intersection power of $-K_X$ shows that it is big. Thus, by the Basepoint-free Theorem (see \cite[Theorem 3.3]{KM98}), $-K_X$ is ample.
\end{proof}

With $X$ Fano, it is natural to question whether $X$ is K-semistable. In this case, a quick computation of a specific $\beta$-invariant shows that such $X$ is always K-unstable. This was previously shown in \cite{ZZ22}; however, we include the computation of the $\beta$-invariant here as it is indicative of the general procedure we will apply in Section~\ref{other} to show various blow-ups of $X$ are also K-unstable. 

\begin{prop}
    $X$ is $K$-unstable, with $\beta(V_0) < 0$.
\end{prop}
\begin{proof}
\begin{align*}
    \beta_X(V_0) = A_X(V_0) - S_X(V_0) = 1 - S_X(V_0)
\end{align*}
Thus it suffices to show that $S_X(V_0) > 1$. We note that $V_0$ has pseudo-effective threshold $\tau = 2$, and that for $ 0 \leq t \leq 2$, $-K_X - t V_0$ is nef, thus:
\begin{align*}
    S_X(V_0)    &= \frac{1}{\vol(X)} \int_0^2 (-K_X -t V_0)^n dt \\
                &= \frac{1}{\vol(X)} \int_0^2 \frac{\vol(V)}{r^{n-1}} \left( (r+1)^n - (r+t-1)^n \right) dt \\
                &= \frac{\vol(V)}{r^{n-1} \vol(X)} \left( 2(r+1)^n - \frac{(r+1)^{n+1} - (r-1)^{n+1}}{n+1} \right) \\
                &= \frac{1}{(r+1)^n - (r-1)^n} \left( 2(r+1)^n - \frac{(r+1)^{n+1} - (r-1)^{n+1}}{n+1} \right) \\
                &= \frac{1}{(n+1)((r+1)^n - (r-1)^n)} \left(f(r-1) - f(r+1)\right) + 1 \\
\end{align*}
Where $f(x) = x^{n+1} -(n+1)x^n$ which is strictly decreasing on the range $ 0< x < n$, thus showing the first term is positive and thus $S_X(V_0) > 1$ as desired.
\end{proof}

Due to the $\P^1$-bundle structure of $\phi: X \to V$, we have a natural $\bT$-action on $X$ where $\bT \cong \G_m$ acts fibre-wise, and on each fiber of $\phi$, $\bT$ acts with the standard $\G_m$-action on $\P^1$. Given this action, $X$ has two divisors that are point-wise fixed by $\bT$: $V_0$ and the image of a positive section $s_\infty: V \to X$ of the $\P^1$-bundle structure, which we will denote as $V_\infty$ (hereafter also referred to as the infinity section).

We continue with the construction. We label the the image of the divisor $B$ under $s_\infty$ as $B_\infty$. This image is a codimension $2$ subvariety of $X$ isomorphic to $B$ since $\s_\infty$ is an embedding. We denote the blow-up of $X$ along $B_\infty$ as $\pi: Y = \Bl_{B_\infty} X \to X$, and the exceptional divisor as $E$. For ease of notation, we also denote the strict transform of the pullback of $B$ along $\phi$ as $F := \phi^{-1}_*(B)$ and the strict transform of $V_\infty$ as $\bVi := \phi^{-1}_*(V_\infty)$. By abuse of notation, we will denote the pullbacks $\pi^*(H)$ and $\pi^*(V_0)$ as $H$ and $V_0$ respectively. We note that $Y \to V$ has the structure of a conic bundle, and the fibers are reducible conics precisely over $B \subset V$. We note that, since $B_\infty$ is fixed under the $\bT$-action on $X$, said action lifts to an action on $Y$, with the locus of fixed points of the $\bT$-action is precisely $\bVi \sqcup V_0 \sqcup (E \cap F)$.

\begin{lemma}
     $Y$ is a Fano variety for $0 \leq l < r+1$ (where we interpret $l = 0$ to mean $Y = X$). Moreover, the anti-canonical volume of $Y$ is
        \[
        \vol(Y) = \begin{cases}
                    \left(\frac{r^n -(r+1-l)^n}{l-1} + r^n -(r-1)^n\right)\frac{\vol(V)}{r^{n-1}}, &\text{if } l\neq 1  \\
                    \left( nr^{n-1} + r^n -(r-1)^n\right)\frac{\vol(V)}{r^{n-1}},  &\text{if } l=1   
                    \end{cases}
        \]
    
    In particular, when $ l = 2$, the anti-canonical volume of $Y$ is 
    \[
    \vol(Y) = \frac{2(r^n - (r-1)^n)}{r^{n-1}} \vol(V).
    \]
\end{lemma}
\begin{proof}
    \begin{align*}
        (-K_Y)^n    &= (2 \pi^* H + \frac{r-1}{r} \pi^*\phi^*(-K_V) -E)^n \\
                    &= (\pi^* H + \frac{r-1}{r}\pi^*\phi^*(-K_V) + \bVi )^n \\
                    &= (\pi^* V_0 + \pi^*\phi^*(-K_V) +\bVi)^n\\
                    &= (\pi^* V_0 + \pi^*\phi^*(-K_V))^n + (\pi^*\phi^*(-K_V) + \bVi)^n -(\pi^*\phi^*(-K_V))^n \\
                    &= (V_0 + \pi^*\phi^*(-K_V))^n + (\pi^*\phi^*(-K_V) + \bVi)^n \\
    \end{align*}
    with $V_0, \bVi$ as denote above. Now, computing monomial products, we have
    \begin{align*}
        V_0^k \cdot \pi^*\phi^*(-K_V)^{n-k}  &= V_0^{k-1} \cdot V_0 \cdot \pi^*\phi^*(-K_V)^{n-k} \\
        &= (H - \frac{1}{r} \pi^*\phi^*(-K_V))^{k-1} \cdot V_0 \cdot \pi^*\phi^*(-K_V)^{n-k} \\
        &= (- \frac{1}{r} \pi^*\phi^*(-K_V))^{k-1} \cdot V_0 \cdot \pi^*\phi^*(-K_V)^{n-k} \\
        &= (\frac{-1}{r})^{k-1} \vol(V) \\
    \end{align*}
    and, similarly,
    \begin{align*}
        \bVi^k \cdot \pi^*\phi^*(-K_V)^{n-k}  &= \bVi^{k-1} \cdot \bVi \cdot \pi^*\phi^*(-K_V)^{n-k} \\
        &= (V_0+\frac{1}{r}\pi^*\phi^*(-K_V) -E)^{k-1} \cdot \bVi \cdot \pi^*\phi^*(-K_V)^{n-k} \\
        &= (\frac{1}{r}\pi^*\phi^*(-K_V) -E)^{k-1} \cdot \bVi \cdot \pi^*\phi^*(-K_V)^{n-k} \\
        &= (\frac{1-l}{r})^{k-1} \vol(V) \\
    \end{align*}
    Substituting yields
    \begin{align*}
        (-K_{Y})^n &= \sum_{i=1}^n \binom{n}{i} \bVi^i \pi^*\phi^*(-K_V)^{n-i} + \sum_{j=1}^n \binom{n}{j} V_0^j \pi^*\phi^*(-K_V)^{n-j} \\
        &= \sum_{i=1}^n \binom{n}{i} (\frac{1-l}{r})^{i-1}\vol(V) + \sum_{j=1}^n \binom{n}{j} (\frac{-1}{r})^{j-1}\vol(V) \\    
    \end{align*}
    \[
     = \begin{cases}
            \left(\frac{r^n -(r+1-l)^n}{l-1} + r^n -(r-1)^n\right)\frac{\vol(V)}{r^{n-1}}, &\text{if } l\neq 1  \\
            \left( nr^{n-1} + r^n -(r-1)^n\right)\frac{\vol(V)}{r^{n-1}},  &\text{if } l=1   
        \end{cases}
    \]

    Note that, under our assumptions on $r$ and $V$, the top intersection power of $-K_Y$ is positive when $0 \leq l < r + 1$. 
    
     Now, we consider the intersection of $-K_Y$ with effective classes of curves on $Y$. In particular, since $-K_Y$ is invariant under the $\bT$-action, we consider intersection products of $-K_Y$ with effective classes of curves fixed by the $\bT$-action. A curve fixed by the $\bT$-action will either be the closure of a single orbit or a curve that is fixed point-wise. A curve that is the closure of a single orbit would either be the strict transform of a fibre from the $\P^1$-bundle structure $X \to V$ or a fibre of the $\P^1$-bundle $E \to B_\infty$. A curve that is fixed point-wise would be contained in the fixed locus of the $\bT$-action, which is $\bVi \sqcup V_0 \sqcup (E \cap F)$. Thus, we consider intersection products of $-K_Y$ with curves $c_0 \in V_0$, curves $c_\infty \in \bVi$, fibres $e \in E$, fibres $f$ of the conic bundle structure over $V \setminus B$, and fibres that lie in $F$ in the class $f-e$. We have the following intersection products:

\begin{align*}    
    -K_Y \cdot c_0 &=  (\pi^* H + \frac{r-1}{r}\pi^*\phi^*(-K_V) + \bVi) \cdot c_0 = \frac{r-1}{r} \pi^*\phi^*(-K_V) \cdot c_0 > 0 \\
    -K_Y \cdot c_\infty &=  (2\pi^* V_0 + \frac{r+1}{r}\pi^*\phi^*(-K_V) - E) \cdot c_\infty = \frac{r-1}{r} \pi^*\phi^*(-K_V) \cdot c_\infty > 0 \\
    -K_Y \cdot f &=  (2H +\frac{r-1}{r} \pi^*\phi^*(-K_V) -E) \cdot f = 2H \cdot f  = 2 > 0 \\
    -K_Y \cdot e &=  (2H +\frac{r-1}{r} \pi^*\phi^*(-K_V) -E) \cdot e = - E \cdot e = 1 > 0 \\
    -K_Y \cdot (f - e)  &=  (2H +\frac{r-1}{r} \pi^*\phi^*(-K_V) -E) \cdot (f-e) = 2H - E \cdot (f-e) = 1 > 0 
\end{align*}

Thus, we see $-K_Y$ is strictly nef, and thus the computation of the top intersection power of $-K_Y$ shows that it is big, and thus ample.
\end{proof}

For the next few sections, until Section~\ref{other}, we will work with this construction under the additional assumption that $ l = 2$. We note that there exists an alternate construction of $Y$ when $l = 2$ related to Fano double covers, see \cite[Section 2]{CDGFKMG23}. One property of such $Y$ that we use frequently is the existence of an involution $\iota: Y \to Y$ such that $\iota(\bVi) = V_0$, $\iota(E) = F$, and $\iota$ respects the conic bundle structure of $Y$ over $V$. The existence of such an involution is readily apparent from the point of view of \cite{CDGFKMG23}. From this involution we see that $Y$ permits a second contraction morphism $Y \to X$ distinct from $\pi$: the blow-down of $F$, i.e. the composition $\pi \circ \iota$ of the involution and the blow-up morphism.

With the construction of $Y$ from the triple $(V,B, \cL)$, we now compute certain key $\beta$-invariants. In particular, given the induced $\bT$-action on $Y$, we can consider the test configurations induced by the coweights of the $\bT$-action. Since the torus action on $Y$ is that of a $1$-dimensional algebraic torus, the co-character lattice $N$ is isomorphic to $\Z$, with $1$ the identity co-character. 

\begin{definition}[Futaki character]
    We define the \textit{Futaki character} on $Y$ to be $Fut|_N : \Z \to \R$ where
    \[
    Fut|_N (i) = \beta_Y(\mathrm{wt}_i)
    \]
    with $\mathrm{wt}_i$ the valuation induced by the coweight $i \in \Z$,
\end{definition}

\begin{lemma}\label{futaki}
    On $Y$, we have $\Fut|_N = 0$. In particular, $\beta (\bVi) = \beta(V_0) = 0$.
\end{lemma}

\begin{proof}
The test configuration induced by $\xi$ has $\ord_t$ as its associated valuation under the isomorphism $\C(Y) \cong \C(V)(t)$ induced by the $\bT$-action. This valuation is divisorial with associated divisor $V_0$, so $\Fut(\xi)=\beta_Y(V_0)$. Similarly, the test configuration induced by $-\xi$ has associated divisor $\bVi$. Any other coweight is a multiple of either of these, and $\Fut(b\xi) = b \Fut(\xi)$, it suffices to show $\beta_Y (V_0) = \beta_Y (\bVi) = 0$. We compute both $\beta$-invariants simultaneously. As both $V_0$ and $\bVi$ are prime divisors on $Y$, each has log discrepancy of $1$. Thus it remains to show that $S_Y(V_0) = S_Y(\bVi) = 1$. Let $P_\infty(t), N_\infty(t)$ (resp. $P_0(t), N_0(t)$) be the positive and negative parts of the Zariski decomposition of $-K_Y - t\bVi$ (resp. $-K_Y - tV_0$). Then

\[
P_\infty (t) = 
    \begin{cases}
        -K_Y - t \bVi = (1-t)\bVi +\pi^*\phi^*(-K_V) + V_0,   &\text{if } 0 \leq t \leq 1 \\
        (2-t)H +\frac{r-1}{r} \pi^*\phi^*(-K_V) 
        =\frac{r+1-t}{r} \pi^*\phi^*(-K_V)+(2-t)V_0, &\text{if } 1 \leq t \leq 2
    \end{cases}
\]

\[
N_\infty (t) = 
    \begin{cases}
        0,                                       &\text{if } 0 \leq t \leq 1 \\
        (1-t)E,                                  &\text{if } 1 \leq t \leq 2
    \end{cases}
\]

\[
P_0 (t) = 
    \begin{cases}
        -K_Y - t V_0 = \bVi + \pi^*\phi^*(-K_V) + (1-t)V_0,  &\text{if } 0 \leq t \leq 1 \\
         (2-t)\bVi +\frac{r+1-t}{r} \pi^*\phi^*(-K_V),       &\text{if } 1 \leq t \leq 2
    \end{cases}
\]

\[
N_0 (t) = 
    \begin{cases}
        0,                                &\text{if } 0 \leq t \leq 1 \\
        (t-1) F = (t-1) (\bVi +\frac{1}{r}\pi^*\phi^*(-K_V) - V_0), 
        &\text{if } 1 \leq t \leq 2
    \end{cases}
\]

This can be seen as follows: for the range $0 \leq t \leq 1$, $-K_Y - t\bVi$ and $-K_Y - tV_0$ are both nef, thus the negative parts of their Zariski decompositions are trivial. For $-K_Y - t\bVi$ with $1 \leq t \leq 2$, we observe that $P_\infty$ is the pullback along the contraction $\pi$ of a nef class on $X$ and $N_\infty$ is supported on the exceptional locus of the same contraction. Thus by \cite[Proposition 2.13]{Okawa16}, we see that this is the Zariski decomposition of  $-K_Y - t\bVi$. A similar argument shows that $P_0 + N_0$ is the Zariski decomposition of  $-K_Y - tV_0$, using the contraction $\pi': Y \to X$ that contracts $F$ (this contraction is $\pi \circ \iota$).

We observe that $\iota (P_\infty) = P_0$ and $\iota (N_\infty) = N_0$. From this and the computations that $V_0 \cdot \bVi = 0$ and $V_0^k \cdot \pi^*\phi^*(-K_V)^{n-k} = \bVi^k \cdot \pi^*\phi^*(-K_V)^{n-k}$, we conclude that $S_Y(V_0) = S_Y(\bVi)$, and thus we need only compute one of them (this can also be more directly seen from $\iota(V_0) = \bVi$). 

\begin{align*}
    S_Y(\bVi) &= \frac{1}{\vol(Y)} \int_0^\infty \vol(-K_Y -t\bVi) dt \\
    &= \frac{1}{\vol(Y)} \int_0^2 \left(P_\infty(t)\right)^n dt \\
    &= \frac{1}{\vol(Y)} \int_0^1 ((1-t)\bVi + \pi^*\phi^*(-K_V) + V_0)^n dt \\
    &+ \frac{1}{\vol(Y)} \int_1^2 (\frac{r+1-t}{r}\pi^*\phi^*(-K_V)+(2-t)V_0)^n dt \\
    &= \frac{1}{\vol(Y)} \int_0^1 \frac{2r^n - (r-1)^n - (r+t-1)^n}{r^{n-1}} \vol(V) dt\\
    &+ \frac{1}{\vol(Y)} \int_1^2 \frac{(r+t-1)^n - (r-1)^n}{r^{n-1}} \vol(V) dt \\
    &= \frac{1}{\vol(Y)} \left( \frac{(r-1)^{n+1} -r^{n+1}}{n+1} +2r^n -(r-1)^n \right) \frac{\vol(V)}{r^{n-1}} \\
    &+ \frac{1}{\vol(Y)} \left( \frac{r^{n+1} -(r-1)^{n+1}}{n+1} -(r-1)^n \right) \frac{\vol(V)}{r^{n-1}} \\
    &= 1
\end{align*}

We note that there is a second, simpler proof. Since $\beta_Y(bv) = b\beta_Y(v)$. By the linearity of the Futaki character we have that $\beta_Y(V_0) = \beta_Y(\mathrm{wt}_1) = -\beta_Y(-\mathrm{wt}_1)= -\beta_Y(\mathrm{wt}_{-1}) = -\beta_Y (\bVi)$, and that $\beta_Y (V_0) = \beta_Y(\bVi)$.
\end{proof}

\section{Proof of Main Theorem}

\subsection{K-Semistability}
 In the light of Theorem~\ref{equiv} and the natural $\bT$-action on $Y$, we investigate the $\bT$-equivariant K-semistability of $Y$ in relation to the K-semistability of $(V,aB)$.
\subsubsection{Reverse Implication}

\begin{definition}(vertical and horizontal divisors \cite[Definition 1.8]{Cheltsov08}) 
    Let $D$ be a $\bT$-invariant divisor over $Y$. $D$ is called a \textit{vertical divisor} if the maximal $\bT$-orbit in $D$ has the same dimension as $\bT$, and is otherwise called a \textit{horizontal divisor}.
\end{definition}

\begin{lemma} \label{vertical}
    For $D$ a vertical divisor over $Y$, there exists a vertical divisor $D'$ over $Y$ such that $\beta_Y (D) = \beta_Y (D')$ and $c_Y(D') \cap \bVi$ is nonempty.
\end{lemma}
\begin{proof}
Let $D$ be a vertical divisor over $Y$. Then $c_Y (D)$ is a $\bT$-invariant subvariety of $Y$.  Either the closed points of $c_Y (D)$ are all fixed by the $\bT$-action or $c_Y (D)$ consists of closures of orbits of points not fixed by $\bT$. In the later case, these closures of orbits are either fibres of $E$ thought of as a $\P^1$-bundle over $B_\infty$, pullbacks of fibres of $X \to V$ away from $B \in V$, or the strict transform of fibres of $X \to V$ over $B$. If $c_Y (D)$ contains any fibres from the first two cases, these fibres, and thus $c_Y (D)$, have nonempty intersection with $\bVi$, so we set $D' := D$. If this is not the case, then $c_Y (D) \subset F$. From here, we consider the divisorial valuation $v' := \ord_D \circ \iota$ and let $D'$ be a divisor associated to the valuation $v' $. Then $c_Y (D') = \iota (c_Y (D)) \subset E$ and $c_Y (D')$ consists of fibres of $E \to B_\infty$, so $c_Y(D') \cap \bVi$ is nonempty. Since $\ord_{D'} = \ord_D \circ \iota$, $\beta_Y(D) = \beta_Y(D')$.

If all closed points $y \in c_Y (D)$ are fixed by the $\bT$-action, then we have that $c_Y (D)$ is either contained in $\bVi$, $V_0$, or $E \cap F$. If $c_Y (D) \subset \bVi$, then take $D' := D$. If $c_Y (D) \subset V_0$, we consider the divisorial valuation $v' := \ord_D \circ \iota$ and let $D'$ be a divisor associated to the valuation $v'$. Then $c_Y (D') = \iota (c_Y (D))$ and $c_Y (D') \cap \bVi = \iota (c_Y (D) \cap V_0) \neq \emptyset$. Again, since $\ord_{D'} = \ord_D \circ \iota$, $\beta_Y(D) = \beta_Y(D')$.

Suppose $c_Y(D) \subset E \cap \iota(E)$, and denote $\ord_D$ as $v$. Let $\mu := \ord_D |_{\C(V)}$ where the inclusion $\C(V) \subset \C(Y)$ is induced by the composition $\pi \circ \phi$. 

Let $\mu$ be a valuation on $V$ and $c_V (\mu) \subset U \subset V$ be a Zariski-open neighborhood of $c_V (\mu)$. Choose a trivialization of $\cL$ over $U$. This trivialization gives a birational map $X \dashrightarrow V \times \A^1$, and composing by $\pi$ gives us a birational map $Y \dashrightarrow V \times \A^1$. This birational map induces an isomorphism on function fields, so we have $\C(Y) \cong \C(V \times \A^1) \cong \C(V)(t)$. 

We define the valuation $v_{\mu,b}$ on $Y$ for $\mu$ a valuation on $V$ and $b \in \Z$ as follows:
\[
v_{\mu,b} (f) = \min_{i \in \Z} \{\mu(f_i) + bi\}
\]
where $f = \sum_i f_i t^i$ is the weight decomposition under the $\G_m$ action via the above isomorphism of function fields. Moreover, every $\bT$-invariant valuation on $Y$ is of this form for some (not necessarily unique) choice of $\mu,$ $U,$ and $b$. Indeed, suppose $v$ is a $\bT$-invariant valuation and let $\mu = v \circ \overline{s_0}$ where $\overline{s_0}: V \to Y$ is the lifting of the zero section $s_0 : V \to X$. Choose $U$ such that $\cL$ has a trivialization over $U$ and $c_V (\mu) \subset U$. Let $b = v(t)$ under the isomorphism $\C(Y) \cong \C(V)(t)$ induced by the trivialization. Then, for $f = \sum f_i t^i \in \C(Y)$, 
\begin{align*}
    v(f) = \min_{i \in \Z} \{v(f_i t^i)\} = \min_{i \in \Z} \{v(f_i) + bi\} =  v_{\mu,b} (f)
\end{align*}

 We then have by \cite[Proposition 3.12]{Li22},
\[
\beta(v_{b}) = \beta(v) + b\Fut(\xi) = \beta(v) = \beta(\mu^*),
\]
with $\mu^* = v_{\mu,0}$. The second equality follows from Lemma~\ref{futaki}. Now, we note that $c_Y(\mu^*) = \phi^{-1}_*(c_V(\mu))$ and thus consists of $1$-dimensional orbits of the $\bT$-action. If $c_Y(\mu^*) \cap \bVi$ is nonempty, we take $D'$ to be a divisor over $Y$ such that $D'$ is the associated divisor to the divisorial valuation $\mu^*$. If $c_Y(\mu^*) \cap \bVi$ is empty, then $c_Y (\mu^*) \subset F$, so we consider the valuation $\mu^{* '} := \mu^* \circ \iota$, and take $D'$ to be a divisor associated to $\mu^{* '}$.  From there we see that $\beta_Y(D') = \beta_Y(\mu^*)=\beta_Y(D)$.
\end{proof}

\begin{lemma} \label{horizontal}
    The only horizontal divisors over $Y$ are $\bVi$ and $V_0$. 
\end{lemma}

\begin{proof} Any horizontal divisor must solely consist of points fixed by the torus action since the maximal orbit contained in the divisor is strictly less than 1 and thus must be zero dimensional. Thus, the only horizontal $\bT$-invariant divisors on $Y$ are $\bVi$ and $V_0$, as all fixed points are either contained with these two divisors or in the intersection $E \cap F$.

A nontrivial horizontal valuation $v$ when restricted to $\C(Y)^{\bT}$ will be the trivial valuation, and thus will be equal to $\mathrm{wt}_b$ for some $b\in \Z, b \neq 0$. If $b > 0$, then the associated divisor will be $V_0$, and if $b < 0$ then the associated divisor will be $\bVi$.
\end{proof}

In order to compare $\delta$-invariants using the techniques developed in \cite{AZ22}, we want to consider the refinement of the linear system $T_\bullet$ where $ T_m = H^0(Y, -mK_Y)$ by the divisor $\bVi$, which we will denote as $W_{\bullet,\bullet}$. This is defined (see Definition~\ref{refine}) as 
\[
W_{m,j} := \text{Im} \left(H^0 (Y,-m K_Y - j \bVi) \xrightarrow{\rho} H^0 (\bVi, (-m K_Y - j \bVi)|_{\bVi}) \right)
\]
where the map $\rho$ is induced by the inclusion of $\bVi$ into $Y$. 

\begin{lemma}\label{refinement}
The refinement of $T_\bullet$ by $\bVi$ is 
\[
W_{m,j} =
    \begin{cases}
        H^0 (\bVi, (-m K_Y - j \bVi)|_{\bVi}), & \text{if } j \leq m\\
        (j-m)B + H^0 (\bVi, P(m,j)|_{\bVi}), & \text{if } m < j \leq 2m\\
        0, & \text{if } j > 2m
    \end{cases}
\]
The movable part of $W_{m,j}$, denoted as $M_{m,j}$, is
\[
M_{m,j} =
    \begin{cases}
        H^0 (V, -m K_V - ( m - j) \cL ), & \text{if } j \leq m\\
        H^0 (V, -mK_V + (m - j ) \cL), & \text{if } m < j \leq 2m\\
        0, & \text{if } j > 2m
    \end{cases}
\]
and the fixed part of $W_{m,j}$, denote as $F_{m,j}$, is
\[
F_{m,j} = 
    \begin{cases}
        0, & \text{if } j \leq m \\
        (j-m)B, & \text{if } n < j \leq 2m \\
        0, & else
    \end{cases}
\]
\end{lemma}
\begin{proof}
For $j > 2m$, $-mK_Y - j\bVi$ is no longer psuedo-effective, and thus the image of $\rho$ is $0$.

Denoting the positive (resp. negative) part of the Zariski decomposition of the divisor $-mK_Y - j\bVi$ with $P(m,j)$ (resp. $N(m,j)$), we have
\[
P(m,j) = 
    \begin{cases}
        -mK_Y - j\bVi,    & \text{if } j \leq m\\
        -mK_Y - j\pi^*(V_{\infty}),   & \text{if } m < j \leq 2m 
    \end{cases}
\]
and
\[
N(m,j) = 
    \begin{cases}
        0,                           & \text{if } j \leq m\\
        (j-m) E,                     & \text{if } m < j \leq 2m
    \end{cases}
\]
Again, we can see that the above is the Zariski decomposition as follows: for $j \leq m$, the divisor is nef, and thus has no negative part. For $ m < j \leq 2m$, $P(m,j)$ is the pullback of a nef divisor along the contraction $Y \to X$ and $N(m,j)$ is supported on the contracted locus. 

To compute $W_{\bullet,\bullet}$ we use the following long exact sequence from the restriction $\rho$:
\begin{align*}
    \dots \rightarrow H^0 (Y, -mK_Y - (j+1) \bVi) \rightarrow H^0 (Y,-m K_Y - j \bVi) \xrightarrow{\rho} \\ 
    H^0 (\bVi, (-m K_Y - j \bVi)|_{\bVi}) \rightarrow H^1 (Y, -mK_Y - (j+1) \bVi) \rightarrow \dots
\end{align*}

Thus, we see that the map on global sections induced by restriction is surjective when $H^1 (Y, -mK_Y - (j+1) \bVi)$ vanishes, which occurs when $ j \leq m$ by the Kawamata-Viehweg Vanishing Theorem. More specifically, when $j \leq m$, $t := \frac{j+1}{m+1} \leq 1$ and $-K_Y - t \bVi$ is both big and nef, and thus so is $(m+1)(-K_Y - t \bVi) = -(m+1) K_Y - (j+1) \bVi$, so by Kawamata-Viehweg Vanishing Theorem, $H^i (Y, -mK_Y - (j+1) \bVi) = 0$ for $i > 0$.

Now we consider when $m < j \leq 2m$. In this case, the restriction of the negative part of the Zariski decomposition gives us the fixed part of the refinement, which is $(j-m)E_{\bVi} = (j-m) B$. For the movable part, we see that $P(m,j)$ is again big and nef, so $H^1 (Y, P(m,j)) = 0$. Thus, by the analogous long exact sequence from the restriction of $P(m,j)$ to $\bVi$, we see that the restriction map is surjective, implying that the free part of the refinement is the space of global sections of the restriction of $P(m,j)$.
\end{proof}

Note that, since $r\cL \sim_{\Q} -K_Y$, we have
\begin{align*}
-mK_Y|_{\bVi} - j \bVi |_{\bVi} &= -m (K_Y + \bVi)|_{\bVi} - (j-m) \bVi|_{\bVi} \\
&= -m K_{\bVi} - (j - m)(V_{\infty} - E)|_{\bVi} \\
&= -m K_{\bVi} - (j-m)(\cL - B) = mrL -(j-m)(-\cL) \\
&= (mr - m + j) \cL
\end{align*}

When $j \leq m$, this is then the free part of $W_{m,j}$. When $m < j \leq 2m$, we see that the fixed part $F_{m,j}$ of $W_{m,j}$ is $(j-m) B = 2(j-m)\cL$, and thus the free part is $(mr-m+j) \cL - 2(j-m) \cL = (mr+m-j) \cL$.    

Let $N_{m,j}$ denote the dimension of $H^0(V, W_{m,j})$, and $N_m = \sum_j N_{m,j}$. Let
\[
a(n,r) :=  \frac{r^{n+1} - (r-1)^{n+1} - (n+1)(r-1)^n }{2(n+1)(r^n - (r-1)^n)} 
\]
and let $a_m B$ be the fixed part of an $m$-basis type $\Q$-divisor of $W_{\bullet,\bullet}$. Then from the above description of $W_{\bullet,\bullet}$, we have that
\[
 a_m = \frac{1}{m N_m} \sum_{j=0}^{2m} N_{m,j} a_{m,j}
\]

\begin{lemma}\label{anr}
    \[
        \lim_{m\to \infty}a_m(n,r) =  a(n,r)
    \]
\end{lemma}
\begin{proof}
    We note that, asymptotically,
    \begin{align*}
        N_m &=  \sum_{j=0}^m \frac{(mr-m+j)^{n-1} \text{vol}(\cL)}{(n-1)!} \\
            &+ \sum_{j=m+1}^{2m} \frac{(mr+m-j)^{n-1} \text{vol}(\cL)}{(n-1)!} + O(m^{n-2})  \\
            &= \frac{m^{n} \text{vol}(\cL)}{(n-1)!} \left(\sum_{j=0}^m \frac{1}{m}(r-1+\frac{j}{m})^{n-1} \right. \\
            &+ \left. \sum_{j=m+1}^{2m} \frac{1}{m}(r+1-\frac{j}{m})^{n-1} +\frac{(n-1)!}{m^{n} \text{vol}(\cL)}O(m^{n-2})  \right)
    \end{align*}

    Thus, we have 
    \begin{align*}
        \sum_{j=0}^{2m} N_{m,j} a_{m,j} &= \frac{\text{vol}(\cL)}{(n-1)!} \left(\sum_{m+1}^{2m} (mr+m-j)^{n-1} (j-m) +O(m^{n-2}) \right) \\
        &= \frac{m^{n+1} \text{vol}(\cL)}{(n-1)!} \left( \sum_{m+1}^{2m} \frac{1}{m}(r+1-\frac{j}{m})^{n-1} (\frac{j}{m}-1) + \frac{1}{m^n+1}O(m^{n-2}) \right)\\
    \end{align*}
    Thus we compute the limit as follows:
    \begin{align*}
        \lim_{m\to \infty}a_m   &= \lim_{m\to \infty} \frac{1}{m N_m} \sum_{j=0}^{2m} N_{m,j} a_{m,j}  \\
                                &= \lim_{m\to \infty} \frac{\sum_{j=0}^m \frac{1}{m}(r-1+\frac{j}{m})^{n-1}  + \sum_{j=m+1}^{2m} \frac{1}{m}(r+1-\frac{j}{m})^{n-1}}{\sum_{m+1}^{2m} \frac{1}{m}(r+1-\frac{j}{m})^{n-1} (\frac{j}{m}-1)} \\
                                &= \frac{\int _1^2 (r+1-t)^{n-1}(t-1) dt}{2\int_1^2 (r+1-t)^{n-1} dt}  \\
                                &= \frac{1}{2} \left[\frac{-(r-1)^n - \frac{(r-1)^{n+1} -r^{n+1}}{n+1}}{r^n - (r-1)^n}\right]  \\
                                &= a
    \end{align*}
\end{proof}

In the terminology of \cite{AZ22}, $a(n,r)$ is the coefficient of $B$ in the asymptotic fixed part $F(W_{\bullet,\bullet})$ of $W_{\bullet,\bullet}$.
\begin{lemma}\label{localglobaldelta}
    Let $D$ be a divisor over $Y$ and $Z \subset c_Y(D) \subset Y$ be a subvariety of $Y$ contained in the center of $D$ on $Y$. If $\delta_Z(Y) \geq 1$, then $\beta_Y(D) \geq 0$.
\end{lemma}

\begin{proof}
    Suppose $\delta_Z (Y) \geq 1$. Then, by definition of the local $\delta$-invariant, $\frac{A_Y (D')}{S_Y(D')} \geq 1$ for all divisors $D'$ over $Y$ whose center on $Y$ contains $Z$. This includes $D$. Thus, by rearranging terms, $\beta_Y (D) = A_Y (D) - S_Y (D) \geq 0$.  
\end{proof}

\begin{lemma}\label{localdelta}
   Let $Z$ be a subvariety of $Y$ such that $Z \subset \bVi$. Suppose $(V,aB)$ is K-semistable. Then $\delta_Z(Y) \geq 1$.
\end{lemma}

\begin{proof} By \cite[Theorem 3.3]{AZ22}, we have that 
\[
\delta_Z(Y) \geq \text{min} \left\{ \frac{A_Y(\bVi)}{S_Y(\bVi)}, \delta_{Z} (\bVi, W_{\bullet,\bullet}) \right\}
\]
Where $W_{\bullet,\bullet}$ is the refinement of the graded linear series associated to $-K_Y$ by $\bVi$. The first value in the minimum is $1$ by Lemma~\ref{futaki}. Thus, it remains to show that $\delta_{Z} (\bVi, W_{\bullet,\bullet}) \geq 1$.

By the assumption that the pair $(V,aB)$ is K-semistable, we have that $\delta(V,aB) \geq 1$, which in particular means that, for a choice of $\epsilon > 0$, $ \delta_M (V,aB) > 1 -\epsilon$ for all $M$ sufficiently large. So, for any $M$-basis type $\Q$ divisor $\Delta_M$ of $-K_V -aB$, we have that $(V, (1-\epsilon)\Delta_M +aB)$ is lc.

 Since $-K_V - aB = (r-2a) \cL$, for any such $\Delta_M$, $(r-2a)\Delta_M$ is an $M(r-2a)$-basis type $\Q$-divisor of $\cL$, and thus we have $\delta_{M(r-2a)} (V,aB;\cL) >(r-2a)(1-\epsilon)$. 

Now, let $D_m$ be an $m$-basis type divisor of $W_{\bullet,\bullet}$, then
\[
D_m = \frac{1}{m N_m} \sum_{j=0}^{2m} D_{m,j}
\]
where each $D_{m,j}$ is a basis sum divisor of $W_{m,j}$. By our computation of $W_{m,j}$, we see that $D_{m,j} = \Delta_{mr-|m-j|} +a_{m,j}N_{m,j}B$, where $\Delta_{mr-|m-j|}$ is a basis sum divisor of $H^0((mr-|m-j|) \cL)$ and $a_{m,j} = j-m$ for $2m \geq j>m$ and $0$ otherwise.
In light of the previous discussion, we can choose $m$ such that $mr-|m-j|$ is sufficiently large and hence the pair
\[
    \left(V, (r-2a)(1-\epsilon) \frac{1}{(mr-|m-j|) N_{m,j}}\Delta_{mr-|m-j|} + aB\right)
\]
is log canonical for all $0\leq j\leq 2m$. Letting $w_{m,j} = \frac{(mr - |m-j|) N_{m,j}}{m(r-2a_m)N_m}$. Then we see that $\sum_{j=0}^{2m} w_{m,j} = 1$, so by the convexity of log canonicity, we have that the pair
\begin{align*}\label{convexity}
    &\left(V, (r-2a)(1-\epsilon) \sum_{j=0}^{2m} w_{m,j} \left(\frac{1}{(mr-|m-j|) N_{m,j}}\Delta_{mr-|m-j|} + aB\right) \right) \\
    &= \left( V, (1-\epsilon) \frac{1}{m N_m} \left(\sum_{j=0}^{2m} \Delta_{mr-|m-j|} +a_{m,j}N_{m,j}B \right) \right) \\
    &= \left( V, (1-\epsilon) D_m \right)
\end{align*}
is log canonical. Thus, we see that $\delta_{Z} (\bVi, W_{\bullet,\bullet}) \geq 1$ as desired.
\end{proof}

\begin{prop}\label{ssreverse}
    If the pair $(V,aB)$ is K-semistable, then $Y$ is K-semistable.
\end{prop}

\begin{proof}
 We will show that the K-semistability of the pair $(V,aB)$ implies the $\bT$-equivariant K-semistability of $Y$, which then implies the K-semistability of $Y$ by Theorem~\ref{equiv}. Suppose $D$ is a $\bT$-invariant divisor over $Y$. We proceed to show $\beta_Y (D) \geq 0$.
 If $D$ is a horizontal divisor, then $\beta_Y (D) = 0 $ by Lemma~\ref{horizontal} and Lemma~\ref{futaki}. Thus, we suppose $D$ is vertical. Then, by Lemma~\ref{vertical}, we may assume $c_Y(D) \cap \bVi$ is nonempty and thus contains some subvariety $Z$. By Theorem~\ref{localdelta}, $\delta_Y(Z) \geq 1$, and thus $\beta_Y(D) \geq 0$ by Lemma~\ref{localglobaldelta}.
\end{proof}

\subsubsection{Forward Implication} 

First we proceed with a $\G_m$-equivariant strengthening of \cite[Lemma 3.1]{AZ22} and \cite[Proposition 3.2]{AZ22}.

\begin{lemma}\label{basis}
Let $H$ be a finite dimensional vector space over $\C$ with an action of $\G_m$ and two $\G_m$-equivariant filtrations $\mathcal{F}, \mathcal{G}$. Then there exists a $\G_m$-invariant basis of $H$ compatible with both $\mathcal{F}$ and $\mathcal{G}$.
\end{lemma}

\begin{proof}
    The proof mostly follows that of the analogous statement in \cite{AZ22}. Consider the induced filtrations $\mathcal{G}_i$ on $Gr_i^{\mathcal{F}} H$ for each $i$. The equivariance of $\mathcal{F}$ implies that the $\G_m$ action on $H$ restricts to an action on each $Gr_i^{\mathcal{F}} H$; the equivariance of $\mathcal{G}$ implies the equivariance of each $\mathcal{G}_i$ and thus the aforementioned action of $\G_m$ restricts to an action on $Gr_j^{\mathcal{G}_i} Gr_i^{\mathcal{F}} H$ for all $i,j$. 
    For each $i$ and $j$, take a basis $B_{i,j}$ of $\G_m$-eigenvectors of $Gr_j^{\mathcal{G}_i} Gr_i^{\mathcal{F}} H$; for each $i$ lift these bases to a basis $B_i$ of $Gr_i^{\mathcal{F}} H$. Each of these bases will consist of $\G_m$-eigenvectors for the action on $Gr_i^{\mathcal{F}} H$; these then lift to a $\mathcal{F}$-compatible $\G_m$-invariant basis $B$ of $H$. Since $Gr_j^{\mathcal{G}_i} Gr_i^{\mathcal{F}} H = Gr_i^{\mathcal{F}_j} Gr_j^{\mathcal{G}} H$ for all $i,j$, we see that for each $j$, $\bigcup_i B_{i,j}$ forms a basis of $Gr_j^{\mathcal{G}} H$ and that $B$ is the lift of these bases and is thus $\mathcal{G}$-compatible.
\end{proof}

\begin{lemma}\label{equiv2}
    For a linear system $H_{\bullet}$, a filtration $\mathcal{F}$ on $H_{\bullet}$, and a valuation $v$ on $H_{\bullet}$, we have
    \[
    S(H_{\bullet};v) = S_{\G_m} (H_{\bullet}; \mathcal{F};v)
    \]
\end{lemma}

\begin{proof}
    We have that $S_m(H_{\bullet};v) \geq S_{\G_m,m} (H_{\bullet};v)$ by definition. Let $D_m$ be a m-basis type $\Q$-divisor of $H_{\bullet}$ that is $\G_m$-invariant and compatible with $\mathcal{F}$ and $\mathcal{F}_v$; such a $D_m$ exists by Lemma~\ref{basis}. Then $S_m(H_{\bullet};v) = v(D_m) \leq S_{\G_m,m} (H_{\bullet};v)$. Taking the limit of each side as $m$ approaches infinity yields the lemma.
\end{proof}

\begin{lemma}\label{basiscorresp}
    Fix $m$ such that $h^0(Y, -mK_Y) \neq 0$.  There is a one-to-one correspondence between $\bT$-invariant $m$-basis type $\Q$-divisors of $-K_Y$ and $m$-basis type $\Q$-divisors of $W_{\bullet,\bullet}$. 
\end{lemma}

\begin{proof}
    Since there is a clear one-to-one correspondence between $m$-basis type $\Q$ divisors and bases of global sections modulo scaling each basis element, this lemma follows from an analogous correspondence between bases of $R_m := H^0(-mK_Y)$ invariant under the $\bT$-action and sets of bases for $H^0(W_{m,j})$ for $0 \leq j \leq 2m$. We can decompose $R_m = \oplus R_{m,j}$ into the direct sum of eigenspaces under the $\bT$-action, where $\bT$ acts on $s\in R_{m,j}$ by $\sigma(\lambda)s = \lambda^j s$; a $\bT$-invariant basis on $R_m$ is exactly a union of bases on each $R_{m,j}$. In particular, each $R_{m,j}$ is isomorphic to the $j$-th graded piece of the weight filtration on $R_m$, which is exactly the filtration induced by $\mathrm{wt}_{-1} = \ord_{\bVi}$. Thus, by construction, $R_{m,j} \cong H^0(W_{m,j})$, and the correspondence between bases follows.
\end{proof}

With lemmata in hand, we begin our argument for the forward implication of the main theorem by constructing divisors over $Y$ whose $\beta$-invariants are dictated by divisors over $V$.

\begin{convention}\label{condiv}
    Let $D_V$ be a divisor over $V$. Then $D_V$ defines a divisor $D_Y$ over $Y$ as follows:

    Let $W \to V$ be a model on which $D_V$ is a divisor. We consider the birational model $X_W := W \times_V X \xrightarrow{\eta} X$ of $X$, and the divisor $D_X := {\phi_W}^* (D_V)$. $D_X$ defines a divisorial valuation on $X$, which we pullback along the blow-up map $\pi$ to a divisorial valuation $\ord_{D_X}$ on $Y$. We then pull back $\ord_{D_X}$ along $\iota$ to a divisorial valuation $\ord_{D_Y}$ achieved by some divisor $D_Y$ over $Y$.
\end{convention}    

\begin{lemma}
    Under the inclusion $\C(V) \subset \C(Y)$ of function fields induced by the composition $\pi \circ \phi$, for a divisor $D_V$ over $V$ and $D_Y$ as in Convention~\ref{condiv}, we have the identification $\ord_{D_Y}|_{\C(V)} = \ord_{D_V}$. 
\end{lemma}

\begin{proof}
    It is clear that $\ord_{D_X} |_{\C(V)} = \ord_{D_V}$ via the inclusion of function fields induced by $\phi$. Now, $\ord_{D_Y} = \ord_{D_X} \circ \iota$ by construction, so $\ord_{D_Y} = \ord_{D_X}$ on meromorphic functions that are fixed by composition with $\iota$. Since $\iota$ fixes fibres of the conic bundle structure of $\pi \circ \phi: Y \to V$, it in particular fixes $\C(V) \subset \C(Y)$, and the lemma follows.
\end{proof}

\begin{lemma}\label{conventionlogdisc}
    Let $D_V$ be a divisor over $V$ and define $D_Y$ as in Convention~\ref{condiv}. Then $A_Y (D_Y) = A_V (D_V)$.
\end{lemma}
\begin{proof}
First we note that $A_Y(D_Y) = A_Y (\ord_{D_X}) $ since $\ord_{D_X} \circ \iota = \ord_{D_Y}$. Then, we compute $A_Y (D_X)$ as follows:
\begin{align*}
    A_Y(D_X)  &= \ord_{D_X}(K_{-/Y}) + 1 \\ 
            &= \ord_{D_X}(K_{-/X}) - \ord_{D_X}(E) + 1 \\
            &= A_X(D_X) - \ord_{D_X}(E) \\
            &= A_X(D_X) \\
            &= A_V (D_V) 
\end{align*}
where the notation $K_{-/Z}$ for $Z = X,Y$ denotes the relative canonical divisor of an appropriate choice of resolution of $Z$, and the last equality follows from the canonical bundle formula for projective bundles. In particular, $K_{X_W} - \eta^*(K_X) = \phi^*(K_W - f^*K_V)$.
\end{proof}

\begin{prop}\label{conventionbeta}
    For $D_V$ a divisor over $V$ and $D_Y$ as in Convention~\ref{condiv}, $\beta_Y(D_Y) = \beta_{V,aB} (D_V)$.
\end{prop}
\begin{proof}
    Using Lemma~\ref{basis}, let $\Delta_m^{Y}$ be a $\bT$-invariant $m$-basis type $\Q$-divisor of $-K_Y$ compatible with $D_Y$ and $\bVi$. Then $\Delta_m^Y$ decomposes $\bT$-equivariantly as below, with $\Supp(\Gamma)$ consisting of $\bT$-invariant divisors distinct from $\bVi$ (see \cite[Section 3.1]{AZ22}). 
\[
\Delta_m^{Y} = \Gamma + S_m(-K_Y, \bVi) \cdot \bVi
\]
    In particular, $\Gamma|_{\bVi}$ is the $m$-basis type $\Q$-divisor of $W_{\bullet,\bullet}$ corresponding to $\Delta_m^{Y}$ under Lemma \ref{basiscorresp}, i.e. $\Gamma|_{\bVi} = \frac{1}{m N_m} \sum_j (\Delta_{m,j} + a_{m,j} N_{m,j} B)$, where $\Delta_{m,j}$ is a basis-sum divisor of the movable part of the refinement $M_{m,j}$. Since $\Delta_m^{Y}$ is compatible with $D_Y$, each $\Delta_{m,j}$ is compatible with the restriction of $\ord_{D_Y}$ to $\C(V)$, which in particular means that each $\Delta_{m,j}$ is compatible with $D_V$. Since the irreducible components of $\Supp(\Gamma)$ are invariant under the $\bT$ action, we see that 
\[    
    \Gamma = \frac{1}{m N_M} \left( \sum_j  \Gamma_{j} + a_{m,j}  N_{m,j}  E \right)
\]
where $\Gamma_j|_{\bVi} = \Delta_{m,j}$ and each $\Gamma_j$ is strictly supported on the $\bT$-invariant divisors on $Y$. We compute $\ord_{D_Y} (\Gamma_j)$ by decomposing $\Gamma_j$ into its parts supported on specific $\bT$-invariant divisors. Let $D$ denote all divisors on $V$ (other than $B$) such that $\Gamma$ has support on ${\pi}_*^{-1} (\phi^*D)$. We note that $\ord_{D_Y} (\bVi), \ord_{D_Y} (V_0) $ and $ \ord_{D_Y} (F)$ are all $0$: for a choice any one of these three divisors we can find an open neighborhood $U$ of $c_Y(D_Y)$ that has empty intersection with the chosen divisor. Thus, on $U$, the chosen divisor is defined locally by a nonvanishing holomorphic function $f$, in particular $f$ has vanishing order $0$ at $c_Y(D_Y)$. Similarly, we compute $\ord_{D_Y} (E) = \ord_{{\pi}_*^{-1}(D_X)}(F) = \ord_{D_X} (\phi^*(B)) = \ord_V (B)$ and $\ord_{D_Y} ({\pi}_*^{-1} (\phi^*D)) = \ord_{D_V} (D)$. Thus, we have the following computation:

\begin{align*}
    \ord_{D_Y} (\Gamma_j)   &= \coeff_{V_0} (\Gamma_j) \cdot \ord_{D_Y} (V_0) + \coeff_{\bVi} (\Gamma_j) \cdot \ord_{D_Y} (\bVi) \\
                            &+ \coeff_{E} (\Gamma_j) \cdot \ord_{D_Y} (E) + \coeff_{F} (\Gamma_j) \cdot \ord_{D_Y} (F)\\
                            &+ \coeff_{{\pi}_*^{-1} (\phi^*D)} (\Gamma_j) \cdot \ord_{D_Y} ({\pi}_*^{-1} (\phi^*D)) \\
                            &= \coeff_{E} (\Gamma_j) \cdot \ord_{D_Y} (E) + \coeff_{{\pi}_*^{-1} (\phi^*D)} (\Gamma_j) \cdot \ord_{D_Y} ({\pi}_*^{-1} (\phi^*D)) \\
                            &= \coeff_{B} (\Delta_{m,j}) \cdot \ord_{D_V} (B) + \coeff_{D} (\Delta_{m,j}) \ord_{D_V} (D) \\
                            &= \ord_{D_V} (\Delta_{m,j})
\end{align*}
As an aside, we note that $\coeff_{\bVi} (\Gamma_j) = 0$ from the decomposition of $\Delta_m^{Y}$.

With this, we have the following computation of $S_m(D_Y)$:

\begin{align*}
    S_m(D_Y) = \ord_{D_Y} \Delta_m^Y    &= \ord_{D_Y} \left( \Gamma + S_m(-K_Y, \bVi) \cdot \bVi \right) \\
                                        &=  \frac{1}{m N_m} \ord_{D_Y} \left( \sum_j  \Gamma_{j} + a_{m,j}  N_{m,j}  E \right) \\
                                        &+ S_m(-K_Y, \bVi) \cdot \ord_{D_Y}(\bVi) \\
                                        &=  \frac{1}{m N_m} \ord_{D_Y} \left( \sum_j  \Gamma_{j} + a_{m,j}  N_{m,j}  E \right)   \\     
                                        &= \frac{1}{m N_m}  \left( \sum_j  \ord_{D_Y} (\Gamma_{j}) + a_{m,j}  N_{m,j}  \ord_{D_Y} (E) \right) \\
                                        &=  \frac{1}{m N_m}  \left( \sum_j  \ord_{D_V} (\Delta_{m,j}) + a_{m,j}  N_{m,j}  \ord_{D_V} (B) \right) \\
                                        &= \frac{1}{m N_m} \left( \sum_j  m N_{m,j} S_{mr-|m-j|} (D_V)  + a_{m,j} N_{m,j} \ord_{D_V} B \right)
\end{align*}
Taking the limit as $m$ tends to $\infty$, we are left with:
\begin{align*}
    S(D_Y)  &= S(D_V) + a \ord_{D_V} B
\end{align*}

Also by Lemma~\ref{conventionlogdisc} we have $A_V(D_V) = A_Y(D_Y)$, so we have
\begin{align*}
     \beta_{V,aB}(D_V) &= A_{V,aB} (D_V) - S (D_V) \\
        &= A_V(D_V) - a \ord_{D_V} (B) - S (D_V) \\
        &= A_Y (D_Y) - S ( D_Y) \\
        &= \beta_Y (D_Y) 
\end{align*}
\end{proof}
\newpage
\begin{prop}\label{ssforward}
    If $(V,aB)$ is K-unstable, then so is $Y$. 
\end{prop}

\begin{proof}
    Assume that $(V,aB)$ is K-unstable. Let $D_V$ be a destabilizing divisor of the pair $(V,aB)$, that is a divisor on a birational model $W$ of $V$ such that $\beta_{V,aB} (D_V) < 0$. Then by Proposition~\ref{conventionbeta} $\beta_Y (D_Y) < 0$.  
\end{proof}

\subsection{K-Polystability}
By Proposition~\ref{ssforward} and Proposition~\ref{ssreverse}, if either of $(V,aB)$ or $Y$ is not K-semistable, then neither are K-semistable. Since K-polystability implies K-semistability, we assume, for this section, that both of $Y$ and the pair $(V,aB)$ are K-semistable. 

\subsubsection{Forward Implication}
\begin{prop}\label{psforward}
    Suppose $Y$ is K-polystable. Then the log Fano pair $(V,aB)$ is K-polystable.
\end{prop}
\begin{proof}
    
Suppose $Y$ is K-polystable. Let $D_V$ be a divisor over the pair $(V,aB)$ such that $\beta_{V,aB}(D) = 0$. To show $(V,aB)$ is K-polystable it suffices to show that $D_V$ is a product-type divisor. We know that $\beta_Y(D_Y) = \beta_{V,aB}(D_V) = 0$ with $D_Y$ as in Convention~\ref{condiv}, so $D_Y$ is a product-type divisor. Denote this test configuration as $\cY \cong Y\times \A^1$, and denote the $\G_m$-action of the test configuration as $\sigma_{D_Y}: \G_m \times \cY \to \cY$. Further, since $D_Y$ is invariant with respect to the $\bT$-action on $Y$, $\cY$ is a $\bT$-equivariant test configuration.

Consider the closure of the orbit $\overline{\text{orb}}_{\sigma_D}(\bVi \times \{1\}) \subset \cY$ under the action of $\sigma_{D_Y}$ on $\cY$. Let $\cV$ denote this closure. Since $\cY$ is a $\bT$-invariant test configuration, the $\G_m$-action on $Y$ induced by $\sigma_{D_Y}$ commutes with the $\T$-action $\sigma$. We will denote this restriction as $\sigma_{D_Y}^0$. For $t \in \bT, s \in \G_m,$ and $y \in \bVi \subset Y$, we have
\begin{align*}
    \sigma (t) \sigma_{D_Y}^0 (s) (y)  &= \sigma_{D_Y}^0 (s) \sigma (t) (y)  \\
                                        &=  \sigma_{D_Y}^0 (s) (y)
\end{align*}
So $\sigma_{D_Y}^0 (s) (y)$ is in the fixed locus of $\sigma$, which is $\bVi \sqcup V_0 \sqcup (E \cap F)$. However, since $\sigma_{D_Y}^0$ has connected orbits, we see that $\sigma_{D_Y}^0 (s) (y) \in \bVi$. Thus, $\sigma_{D_Y}^0$ maps points in $\bVi$ to points in $\bVi$, which implies that $\sigma_{D_Y}$ maps points in $\bVi \times \A^1 \subset \cY$ to points in $\bVi \times \A^1$. From this we see that $\cV \cong V \times \A^1$ and that $\cV$ with the restriction of the $\G_m$-action $\sigma_{D_Y}$ to $\cV$ forms a product test configuration of $V$. 

The valuation associated to $\cV$ is, by definition, is $\ord_{\cV_0}|_{\C(V)}$. We see that the $\bT$-action $\sigma$ extends to $\cY$, and induces a weight decomposition on $\C(\cY)$ (similar to that on $\C(Y)$) such that $\C(\cY)^\bT \cong \C(\cV)$. From this, we see the $\ord_{\cY_0}|_{\C(\cV)} = \ord_{\cV_0}$. Since $\ord_{D_Y} = b \ord_{\cY_0}|_{\C(\cY)}$ for some $b \in \Z_{>0}$, we have that $b\ord_{\cV_0}|_{\C(V)} = b\ord_{\cY_0}|_{\C(V)} = \ord_{D_Y}|_{\C(V)} =\ord_{D_V}$. Thus, we see that the test configuration associated to $D_V$ is isomorphic to $\cV$ and is thus a product test configuration.
\end{proof}

\subsubsection{Reverse Implication}
\begin{prop}\label{psreverse}
    Suppose the log pair $(V,aB)$ is K-polystable. Then $Y$ is K-polystable.
\end{prop}

\begin{proof}
Suppose that the pair $(V,aB)$ is K-polystable; we seek to show the same for the construction $Y$. In fact, due to \ref{equiv}, it suffices to show that $Y$ is $\bT$-equivariantly K-polystable. 

Suppose $D$ is a $\bT$-invariant divisor over $Y$ such that $\beta_Y(D) = 0$; we seek to show that the test configuration induced by $D$ is a product test configuration. If $D$ is horizontal, then $D = V_0$ or $\bVi$; both of these induce product test configurations. Indeed, $\ord_{V_0} = \mathrm{wt}_\xi$ and $\ord_{\bVi} = \mathrm{wt}_{-\xi}$, so the test configurations induced by $D$ is the product test configuration given by either the $\bT$-action on $Y$ or its inverse.

Now, assume $D$ is vertical. Since the underlying space of two test configurations are isomorphic if one of the associated valuations is achieved by twisting the other by a suitable cocharacter, using Lemma~\ref{vertical} we may assume $c_Y(D) \cap \bVi \neq \emptyset$. 

Let $\mu := \ord_D |_{\C(V)}$, in which case we see via \cite{Li22} that $\ord_D = v_{\mu,n\xi}$ for some $n \in \N$, and denote the divisor associated to $\mu$ as $D_\mu$. Then $\beta_{V,aB}(D_\mu) = \beta_Y(D_{\mu,Y})$, with $D_{\mu,Y}$ as in Convention~\ref{condiv} with respect to $D_\mu$. Since $\ord_{D_{\mu,Y}} = v_{\mu,0}$, we see that $D_{\mu,Y}$ and $D$ differ by a twist and thus $\beta_{V,aB}(D_\mu) = \beta_Y(D_{\mu,Y}) = \beta_Y(D) = 0$. Thus, by our assumption that the log pair $(V,aB)$ is K-polystable, the test configuration associated to $D_\mu$ is a product test configuration.

In what follows, we use a linearization on $\cL$ of the $\G_m$ action $\sigma_\mu$ on $V$ induced by the product test configuration associated to $D_\mu$. Such a linearization exists by \cite[Proposition 2.4]{KKLV89} and its antecedent remark.

Denote the test configuration associated to $\mu$ as $\cV_\mu \cong V \times \A^1$. Lifting $\cL$ along the first projection to $\text{pr}_2^*(\cL) =: \bL$, we can then consider the variety $\cX := \P_{\cV_\mu} (\bL \oplus \cO_{\cV_\mu})$ (which in fact is the test configuration of $X$ associated to some $\bT$-invariant lift of the valuation $\mu$). The aforementioned $\G_m$-linearization of  $\sigma_\mu$ on $\cL$ lifts to a linearization on $\bL \oplus \cO_{V\times A^1}$, thus we have an induced $\G_m$ action on $\cX$ that fixes a positive section of the $\P^1$-bundle $\cX \to V\times\A^1$. Thus, we may define the variety $\cY$ to be the blow-up of $\cX$ along the image of $B \times \A^1$ along the $\bT$-invariant positive section. Thus, the $\G_m$ action lifts to $\cY$, which is in fact a product test configuration for $Y$. 

By construction we have $\C(\cY) \isom \C(Y)(s) \isom \C(V)(t,s)$, and the valuation associated to the test configuration $\cY$ is the restriction $w := \ord_{s=0}|_{\C(Y)}$. The further restriction of $w$ to $\C(V)$ is in fact $\mu$, so $ w = v_{\mu, k\xi}$ for some $k \in \N$, and thus $w$ is equivalent to some twist of $\ord_D$, so their associated test configurations are isomorphic as varieties by \cite[Example 3.7]{Li22}, and thus the test configuration associated to $D$ is a product test configuration.
\end{proof}

\begin{proof}[Proof of Theorem~\ref{main}]
By combining Proposition~\ref{ssforward} and Proposition~\ref{ssreverse}, we have that $(V,aB)$ is K-semistable if and only if $Y$ is K-semistable. Similarly, Proposition~\ref{psforward} and Proposition~\ref{psreverse} together show that $(V,aB)$ is K-polystable if and only if $Y$ is K-polystable.
\end{proof}

With Theorem~\ref{main} established, we now prove Corollary~\ref{ceaconj} through an application of interpolation of K-stability. 

\begin{proof}[Proof of Corollary~\ref{ceaconj}]
  By \cite{LZ22, Zhuang21}, the K-polystability of $W$ is equivalent to that of the pair $(V, \frac{1}{2}B)$. A quick calculation shows that $0 < a(n,r) < \frac{1}{2}$ for all $n,r$, so by interpolation of K-stability (see e.g. \cite[Proposition 2.13]{ADL24}), we have that $(V,aB)$ is K-polystable, so by Theorem~\ref{main}, $Y$ is K-polystable. 
\end{proof}

\section{Examples}\label{sec:examples}
\begin{exa}
    The Fano family \textnumero $3.9$ (previously known, see \cite{ACCFKMGSSV23}): Letting $V = \P^2$, $r = \frac{3}{2}$, and $B$ be a smooth quartic curve, then $\cL$ is $\cO(2)$, $X$ is $\P(\cO(2) \oplus \cO)$, and we have then that by Theorem~\ref{main} $Y$ is a member of the Fano family \textnumero $3.9$ with K-poly/semistability equivalent to that of the pair $(V,\frac{9}{52}B)$. By \cite{ADL24}, we see that such pairs are K-poly/semistable exactly when the quartic plane curve $B$ is poly/semistable in the GIT sense, both of which are implied in this case by the smoothness of $B$. This provides another proof that all smooth members of family \textnumero $3.9$ are K-polystable, as previously shown in \cite{ACCFKMGSSV23}.
\end{exa}   
\begin{exa}
    The Fano family \textnumero $3.19$ (previously known, see \cite{ACCFKMGSSV23}): Letting $V = \P^2$, $r = 3$, and $B$ be a smooth conic, we have then that $\cL = \cO_V (1)$, $X$ is $\P^3$ blown up at a point, and $Y$ is a member of the Fano family  \textnumero $3.19$  with K-poly/semistability equivalent to that of the pair $(V,\frac{33}{152}B)$ by Theorem~\ref{main}. By \cite{LS14}, we see that such a pair K-polystable, thus providing another proof that the unique smooth Fano variety in family \textnumero $3.19$ is K-polystable, originally shown in \cite{IS17}. 
\end{exa}
\begin{exa}
    The Fano family \textnumero $4.2$ (previously known, see \cite{ACCFKMGSSV23}): Let $V = \P^1 \times \P^1$, $r = 2$, and $B$ be a smooth curve of bidegree $(2,2)$. Then, $\cL = \cO_V (1,1)$, $X = \P(\cO(1,1) \oplus \cO)$, and $Y$ is a smooth member of the Fano family \textnumero $4.2$ and, by the above theorem, the K-polystability of $Y$ is equivalent to that of the pair $(V,\frac{11}{56}B)$. The K-polystability of this pair follows from the interpolation of K-stability (see \cite[Proposition 2.13]{ADL24} and \cite[Theorem 2.10]{ADL23}) since $V$ is K-polystable and $(V,B)$ is a plt log Calabi-Yau pair. Thus, by Theorem~\ref{main}, such $Y$ is K-polystable, giving a new proof of this result previously shown in \cite{ACCFKMGSSV23}. 
\end{exa}
\begin{exa}
    New examples from blow-ups related to quartic surfaces in $\P^3$ and higher dimensional analogs: Let $V=\P^3$, $r = 2$, and $B$ a smooth quartic surface in $\P^3$. Then, $\cL=\cO_V (2)$, and $Y$ is the blow-up of the cone over the second Veronese embedding of $\P^3$ in $\P^9$ along the cone point and a quartic surface in the base. In this case, $(V,aB) = (\P^3, \frac{13}{75} B)$. By \cite{ADL23}, such pairs are K-poly/semistable exactly when they are poly/semistable in the GIT sense. As $B$ is smooth, it is GIT polystable, thus all smooth fourfolds constructed in this manner are K-polystable by Theorem~\ref{main}.
    
    We generalize this to higher dimensions: let $V = \P^{n-1}$ for $n$ even, $r=2$, and $B$ a smooth degree $n$ hypersurface in $V$. Then $\cL=\cO_V (\frac{n}{2})$, and $Y$ is the blow-up of the cone over the embedding $\phi_{|\cL|}: V \hookrightarrow \P^N$ along the cone point and the inclusion of $B$ in the base of the cone. By \cite[Theorem 1.4]{ADL24}, there exists some $c_1$ such that $(V,cB)$ is K-semi/polystable if and only if $B$ is semi/polystable in the GIT sense for all $c < c_1$. If $a=a(n,2) < c_1$, then $B$ smooth implies $B$ is polystable in the GIT sense and thus $(V,aB)$ is K-polystable, implying by Theorem~\ref{main} that $Y$ is K-polystable. If $c_1 \leq a$, then again since $B$ is smooth, $(V,cB)$ is K-polystable for some $c < a$. Then, since $(V,B)$ is a plt log  Calabi-Yau pair, $(V,(1-\epsilon)B)$ is K-polystable for some sufficiently small $\epsilon$ (see \cite[Theorem 2.10]{ADL23}), and thus by interpolation of K-stability (see \cite[Proposition 2.13]{ADL24}), we again have that $(V,aB)$ is K-polystable, implying the K-polystability of $Y$ by Theorem~\ref{main}. This improves \cite[Theorem 1.9]{CDGFKMG23}.
\end{exa}
\begin{exa}
    New K-unstable example: Let $V = \Bl_p \P^3$, $r = 2$, and $B$ a smooth member of $|-K_V|$. Then, $\cL = \pi^*(2H) - E$ where $\pi: V \to \P^3$ is the blow-up map and $E$ is the exceptional divisor of the blow-up, $X = \P_V(\cL \oplus \cO_V)$, and $Y$ is $\Bl_{B_\infty} X$. As in the previous example, $a = \frac{13}{75}$.  A straightforward computation shows that the pair $(V,aB)$ is K-unstable with $\beta_{V,aB}(\pi^*(H)) <0$, where $\pi^*(H)$ is the pullback of the hyperplane section of $\P^3$ along the blow-up. Thus, by Theorem~\ref{main}, $Y$ is K-unstable.
\end{exa}

\section{Other Blow-ups of Projective Compactifications of Proportional Line Bundles}
\label{other}
In this section, we see that, when replacing the assumption that $l = 2$ in the construction of $Y$ with $ l \neq 2$ (where $ B \sim_{\Q} l \cL$), the resulting Fano varieties $Y$ are always K-unstable.

\begin{thm}[Theorem~\ref{unstable}]
    Let $Y$ be constructed as above with $V$ and $B$ smooth and $B \sim_\Q l \cL$ for $0 < l < r+1, l \neq 2$. Then $Y$ is K-unstable. Furthermore, either the strict transform of the image of the positive section containing $B_\infty$, denoted as $\bVi$, or the strict transform of the zero section, $V_0$, is a destabilizing divisor for $Y$.
\end{thm}
\begin{lemma}
\begin{align*}
    \beta_{Y} (V_0) + \beta_{Y} (\bVi) = 0 \\
\end{align*}
\end{lemma}
\begin{proof}
As $V_0, \bVi$ are both divisors on $Y$, $A_Y (V_0) = A_{Y} (\bVi) = 1$, so $\beta_{Y} (V_0) + \beta_{Y} (\bVi) = 2 - (S_{Y} (V_0) + S_{Y} (\bVi))$. It remains to show the sum $S_{Y} (V_0) + S_{Y} (\bVi)$ is equal to $2$.

The Zariski decompositions of $-K_{Y} - t\bVi, -K_{Y} - tV_0$ on $Y$ are as follows:
\[
P_\infty (t) = 
    \begin{cases}
        -K_{Y} - t \bVi = (1-t)\bVi +\pi^*\phi^*(-K_V) + V_0,   &\text{if } 0 \leq t \leq 1 \\
        (2-t)H +\frac{r-1}{r} \pi^*\phi^*(-K_V) 
        =\frac{r+1-t}{r} \pi^*\phi^*(-K_V)+(2-t)V_0, &\text{if } 1 \leq t \leq 2
    \end{cases}
\]

\[
N_\infty (t) = 
    \begin{cases}
        0,                                       &\text{if } 0 \leq t \leq 1 \\
        (1-t)E,                                  &\text{if } 1 \leq t \leq 2
    \end{cases}
\]

\[
P_0 (t) = 
    \begin{cases}
        -K_{Y} - t V_0 = \bVi + \pi^*\phi^*(-K_V) + (1-t)V_0,  &\text{if } 0 \leq t \leq 1 \\
         (2-t)\bVi +\frac{r-(t-1)(l-1)}{r} \pi^*\phi^*(-K_V),       &\text{if } 1 \leq t \leq 2
    \end{cases}
\]

\[
N_0 (t) = 
    \begin{cases}
        0,                                                                  &\text{if } 0 \leq t \leq 1 \\
        (t-1)F = (t-1) (\bVi +\frac{l-1}{r}\pi^*\phi^*(-K_V) - V_0),         &\text{if } 1 \leq t \leq 2
    \end{cases}
\]

This can be seen as follows: for the range $0 \leq t \leq 1$, $-K_Y - t\bVi$ and $-K_Y - tV_0$ are both nef, thus the negative parts of their Zariski decompositions are trivial. For $-K_Y - t\bVi$ with $1 \leq t \leq 2$, we observe that $P_\infty$ is the pullback along the contraction $\pi$ of a nef class on $X$ and $N_\infty$ is supported on the exceptional locus of the same contraction. Thus by \cite[Proposition 2.13]{Okawa16}, we see that this is the Zariski decomposition of  $-K_Y - t\bVi$. For $l \neq 1$, a similar argument shows that $P_0 + N_0$ is the Zariski decomposition of  $-K_Y - tV_0$, using the contraction $\pi': Y \to X' := \P_V (\cL' \oplus \cO_V)$ (for some line bundle $\cL'$ such that $l' \cL' \sim_{\Q} -K_V$ where $(l'-1)(l-1) = 1$) that contracts $F$. For $ l = 1$, we have the same contraction map $\pi': Y \to X' := \P_V(\cO_V \oplus \cO_V)$ but now for $\cL' = \cO_V$. 

We have the following intersection products on $Y$, for $k > 0$:
\begin{align*}
        V_0^k \cdot \pi^*\phi^*(-K_V)^{n-k}  &= V_0^{k-1} \cdot V_0 \cdot \pi^*\phi^*(-K_V)^{n-k} \\
        &= (H - \frac{1}{r} \pi^*\phi^*(-K_V))^{k-1} \cdot V_0 \cdot \pi^*\phi^*(-K_V)^{n-k} \\
        &= (- \frac{1}{r} \pi^*\phi^*(-K_V))^{k-1} \cdot V_0 \cdot \pi^*\phi^*(-K_V)^{n-k} \\
        &= (\frac{-1}{r})^{k-1} \vol(V) \\
    \end{align*}
    
    \begin{align*}
        \bVi^k \cdot \pi^*\phi^*(-K_V)^{n-k}  &= \bVi^{k-1} \cdot \bVi \cdot \pi^*\phi^*(-K_V)^{n-k} \\
        &= (V_0+\frac{1}{r}\pi^*\phi^*(-K_V) -E)^{k-1} \cdot \bVi \cdot \pi^*\phi^*(-K_V)^{n-k} \\
        &= (\frac{1}{r}\pi^*\phi^*(-K_V) -E)^{k-1} \cdot \bVi \cdot \pi^*\phi^*(-K_V)^{n-k} \\
        &= (\frac{1-l}{r})^{k-1} \vol(V) \\
    \end{align*}
    
Let us assume $l \neq 1$, then we have for $\vol(Y) S_Y (\bVi)$:
\begin{align*}
    \vol(Y) S_{Y} (\bVi) &= \int_0^1 \left( (1-t)\bVi + \pi^*\phi^*(-K_V) + V_0\right)^n dt \\
    &+ \int_1^2 \left(\frac{r+1-t}{r}\pi^*\phi^*(-K_V) + (2-t) V_0 \right)^n dt\\
    &= \frac{\vol(V)}{r^{n-1}} (\int_0^1 \frac{1}{1-l}\left( (1-t)(1-l) + r)^n -r^n\right) dt \\
    &+\int_0^1 r^n - (r-1)^n dt - \int_1^2 (r-1)^n - (r+1-t)^n dt ) \\
    &= \frac{\vol(V)}{r^{n-1}} \int_0^1 \frac{((1-t)(1-l) + r)^n -r^n}{1-l} + r^n - 2(r-1)^n +(r+t-1)^n dt 
\end{align*}

Similarly, for $\vol(Y) S_Y (\bVi)$:

\begin{align*} 
    \vol(Y) S_{Y} (V_0) &= \int_0^1 \left( \bVi + \pi^*\phi^*(-K_V) + (1-t) V_0\right)^n dt \\
    &+ \int_1^2 \left((2-t)\bVi +\frac{r-(t-1)(l-1)}{r} \pi^*\phi^*(-K_V)\right)^n dt\\
    &= \frac{\vol(V)}{r^{n-1}} (\int_0^1 \frac{1}{1-l}\left( (r+1-l)^n -r^n\right) dt + \int_0^1 r^n - (r+t-1)^n dt \\
    &+ \int_1^2  \frac{1}{1-l}\left( (r+1-l)^n -(r-(t-1)(l-1))^n \right) dt ) \\  
    &= \frac{\vol(V)}{r^{n-1}} \int_0^1  \frac{(r+1-l)^n -r^n}{1-l} + r^n - (r+t-1)^n \\
    &+ \frac{ (r+1-l)^n - (r+(t-1)(l-1))^n }{1-l}dt 
\end{align*}

Summing these two terms we get precisely $2 \vol(Y)$, and so $\beta(V_0) +\beta(\bVi) = 0$.

Now, for $ l = 1$, we have for $\vol(Y) S_Y (\bVi)$:

\begin{align*}
    \vol(Y) S_{Y} (\bVi) &= \int_0^1 \left( (1-t)\bVi + \pi^*\phi^*(-K_V) + V_0\right)^n dt \\
    &+ \int_1^2 \left(\frac{r+1-t}{r}\pi^*\phi^*(-K_V) + (2-t) V_0 \right)^n dt\\
    &= \frac{\vol(V)}{r^{n-1}} (\int_0^1 n(1-t)r^{n-1} + r^n - (r-1)^n dt \\
    &- \int_1^2 (r-1)^n - (r+1-t)^n dt ) \\
    &= \frac{\vol(V)}{r^{n-1}} \int_0^1 n(1-t)r^{n-1} + r^n - 2(r-1)^n +(r+t-1)^n dt 
\end{align*}

and similarly for $\vol(Y) S_Y (V_0)$:
\begin{align*} 
    \vol(Y) S_{Y} (V_0) &= \int_0^1 \left( \bVi + \pi^*\phi^*(-K_V) + (1-t) V_0\right)^n dt \\
    &+ \int_1^2 \left((2-t)\bVi +\frac{r-(t-1)(l-1)}{r} \pi^*\phi^*(-K_V)\right)^n dt\\
    &= \frac{\vol(V)}{r^{n-1}} (\int_0^1 nr^{n-1} + r^n - (r+t-1)^n dt + \int_1^2 n(2-t)r^{n-1} dt) 
\end{align*}

So for the $l = 1$ case, summing these two terms we again get $2 \vol(Y)$, and so $\beta(V_0) +\beta(\bVi) = 0$.
\end{proof}

Thus, to finish the prove of Theorem~\ref{unstable}, we simply must show that $\beta_{Y} (\bVi) \neq 0$ for $l \neq 2$. 
\begin{proof}
In fact, we show that $\frac{r^{n-1}\vol(Y)}{\vol(V)} \beta_{Y} (\bVi) \neq 0$.
\begin{align*}
    \vol(Y) S_{Y} (\bVi) &= \int_0^1 \left( (1-t)\bVi + \pi^*\phi^*(-K_V) + V_0\right)^n dt \\
    &+ \int_1^2 \left(\frac{r+1-t}{r}\pi^*\phi^*(-K_V) + (2-t) V_0 \right)^n dt\\
    &= \frac{\vol(V)}{r^{n-1}} (\int_0^1 \frac{1}{1-l}\left( (1-t)(1-l) + r)^n -r^n\right) dt \\
    &+\int_0^1 r^n - (r-1)^n dt + \int_1^2 (r-1)^n - (r+1-t)^n dt ) \\
    &= \int_0^1 \frac{((1-t)(1-l) + r)^n -r^n}{1-l} + r^n - 2(r-1)^n +(r+t-1)^n dt 
\end{align*}

\begin{align*}
    \frac{r^{n-1}\vol(Y)}{\vol(V)} \beta_{Y} (\bVi) &= \frac{r^{n-1}\vol(Y)}{\vol(V)} - \frac{r^{n-1}\vol(Y)}{\vol(V)} S_{Y} (\bVi) \\
    &= \left(\frac{r^n -(r+1-l)^n}{l-1} + r^n -(r-1)^n\right) \\
    &- \int_0^1 \frac{((1-t)(1-l) + r)^n -r^n}{1-l} + r^n - 2(r-1)^n +(r+t-1)^n dt \\ 
    &= \int_0^1 \left(\frac{r^n -(r+1-l)^n}{l-1} + r^n -(r-1)^n\right) \\
    &-  \frac{((1-t)(1-l) + r)^n -r^n}{1-l} - r^n + 2(r-1)^n - (r+t-1)^n dt \\ 
    &= \int_0^1 \frac{(r - (t-1)(1-l) )^n -(r+1-l)^n}{l-1}  + (r-1)^n - (r+t-1)^n dt 
\end{align*}
We will show the integrand is strictly positive (resp. negative) when $r > 2$ (resp. $r < 2$) for $ 0 < t < 1$. Let $f(x) = (r-x)^n$, $w = 1-t$ and $ y = l-1$. Then the integrand is equal to 
\[
    (1-w) \left( \frac{f(1)-f(w)}{1-w} - \frac{f(wy) - f(y)}{wy - y} \right)
\]
Then, since $f$ is strictly convex on the interval $(0,r)$, we have, for $l > 2$, 
\[
      \frac{f(1)-f(w)}{1-w} < \frac{f(1) - f(wx)}{1-wx} < \frac{f(wy) - f(y)}{wy - y} 
\]
with the inequalities reversing for $l < 2$. Thus the integrand is strictly positive (resp. negative), so $\beta_{Y}(\bVi) \neq 0$ as desired.
\end{proof}

\subsection{Examples of Other Blow-ups}
\begin{exa}
    The Fano family \textnumero $3.14$: Letting $V$ be $\P^2$, $r = 3$, and $B$ a smooth planar cubic curve, so that $l = 3$, we have that $X = \Bl_p \P^3$ is the blow-up of $\P^3$ at a point and $Y$ is a smooth member of the family of Fano threefolds \textnumero $3.14$, and all smooth members of said family are obtained in this way. Thus, Theorem~\ref{unstable} recovers the K-unstability of members of this family, originally due to Fujita \cite[Theorem 1.4]{Fujita16}.  
\end{exa}
\begin{exa}\label{exauns}
    K-unstable in each dimension: Let $V$ be $\P^{n-1}$, $r = n$, and $B$ a degree $d$ hypersurface in $V$, for $d < n$. Then $X = \Bl_p \P^n$ the blow up of $\P^n$ at a point and $Y$ is the blow-up of projective $n$-space along a codimension $2$ subvariety contained in the pullback of a hyperplane in $X$ that doesn't contain $p$. Such $Y$ is K-unstable by Theorem~\ref{unstable}. Thus, we have for every $n$ several examples of a K-unstable Fano variety of dimension $n$.
\end{exa}

\bibliographystyle{alpha}
\bibliography{ref}

\end{document}